\documentclass[11pt]{amsart} 

\usepackage[utf8]{inputenc} 

  \usepackage{fullpage}

\usepackage{graphicx} 

\usepackage{amsmath}
\usepackage{amssymb}
\usepackage{color}
\usepackage{amstext}
\usepackage{amsthm}
\usepackage{url}
\usepackage{booktabs} 
\usepackage{array} 
\usepackage{paralist} 
\usepackage{verbatim} 
\usepackage{subfig} 

\usepackage{fancyhdr} 

\usepackage{listings}
\usepackage{epstopdf}

\title{A uniformness conjecture of the Kolakoski sequence, graph connectivity, and correlations}
\newcommand\blfootnote[1]{%
  \begingroup
  \renewcommand\thefootnote{}\footnote{#1}%
  \addtocounter{footnote}{-1}%
  \endgroup
}

\begin{document}
\maketitle
{\centering Bobby Shen\footnote{Bobby Shen,  Department of Mathematics, Massachusetts Institute of Technology,
Cambridge, Massachusetts, USA, \url{runbobby@mit.edu}} \par}

\newtheorem{theorem}{Theorem}[section]
\newtheorem{corollary}{Corollary}[theorem]
\newtheorem{lemma}[theorem]{Lemma}
\newtheorem{proposition}[theorem]{Proposition}
\newtheorem{conjecture}[theorem]{Conjecture}
 
\theoremstyle{definition}
\newtheorem{definition}[theorem]{Definition}
 
\theoremstyle{remark}
\newtheorem*{remark}{Remark}

\blfootnote{2010 \textit{Mathematics Subject Classsifications}: 60GXX, 11B99, 11B37 }

\begin{abstract}
The Kolakoski sequence is the unique infinite sequence with values in $\{1,2\}$ and first term $1$ which equals the sequence of run-lengths of itself; we call this $K(1,2).$ We define $K(m,n)$ similarly for $m+n$ odd. A well-known conjecture is that the limiting density of $K(1,2)$ is one-half. We state a natural generalization, the ``generalized uniformness conjecture'' (GUC). The GUC seems intractable, but we prove a partial result: the GUC implies that members of a certain family of directed graphs $G_{m,n,k}$ are all strongly connected. We prove this unconditionally.

For $d>0,$ let $cf(m,n,d)$ be the density of indices $i$ such that $K(m, n)_i=K(m, n)_{i+d}.$ Essentially, $cf(m, n, d)$ is the autocorrelation function of a stationary stochastic process with random variables $\{X_t\}_{t\in\mathbb{Z}}$ whereby a sample of a finite window of this process is formed by copying as many consecutive terms of $K(m,n)$ starting from a ``uniformly random index'' $i\in\mathbb{Z}_+.$ Assuming the GUC, we prove that we can compute $cf(m,n,d)$ exactly for quite large $d$ by constructing a periodic sequence $S$ of period around $10^{8.5}$ such that for $d$ not too large, the correlation frequency at distance $d$ in $K(m,n)$ equals that in $S.$ We efficiently compute correlations in $S$ using polynomial multiplication via FFT. 

We plot our estimates $cf(m,n,d)$ for several small values of $(m,n)$ and $d\le10^5$ or $10^6$. We note many suggested patterns. For example, for the three pairs $(m,n)\in\{(1,2),(2,3),(3,4)\},$ the function $cf(m,n,d)$ behaves very differently as we restrict $d$ to the $m+n$ residue classes $\text{mod}$ $m+n.$ The plots of the three functions $cf(1,2,d),cf(2,3,d),$ and $cf(3,4,d)$ resemble waves which have common nodes. We consider this very unusual behavior for an autocorrelation function. The pairs $(m,n)\in\{(1,4),(1,6),(2,5)\}$ show wave-like patterns with much more noise.
\end{abstract}

\section{Introduction}
The Kolakoski sequence is the unique infinite sequence with values in $\{1, 2\}$ and first two terms $1, 2, \ldots$ which equals the sequence of run-lengths in the run-length encoding of itself. See \cite{RLE} for our definition of run-length encoding. The existence and uniqueness is relatively easy to prove. The Kolakoski sequence begins $1, 2, 2, 1, 1, 2, 1, 2, 2, 1, 2, 2, 1, 1, 2, \ldots.$ Similarly, for distinct positive integers $m, n,$ we define $K(m,n)$ to be the unique infinite sequence with values in $\{m, n\}$ and first term $m.$ 

The sequences $K(m,n)$ in which $m+n$ is even are less interesting. The main reason is that the infinite sequence $K(m,n)$ is a fixed point of a finite set of substitution rules. We define even-length finite sequences $K_i$ recursively as follows. Let $K_0$ be the sequence $m, m$ if $m>1$ or $m, n$ or $m = 1.$ For $i \ge 0,$ partition $K_i$ into blocks of length 2. Replace $m,m$ with $m^m n^m,$ or the term $m$ repeated $m$ times followed by the term $n$ repeated $m$ times. In order, Replace blocks which are $m, n$ with $m^m n^n.$ Replace $n, n$ with $m^n n^n.$ One can show that the block $n, m$ does not occur. Let $K_{i+1}$ be the new sequence. Note that since $m+n$ is even, the substitution rules turns one block into a finite number of blocks. 

For example, if $m = 1, n = 3,$ then $K_0 = 1, 3.$ $K_1 =1, 3, 3, 3.$ $K_2 = 1, 3, 3, 3, 1, 1, 1, 3, 3, 3.$ 

One can show that $K_i$ agrees with the first $|K_i|$ terms of $K_{i+1}$ and also the first $|K_i|$ terms of $K_{m, n}.$ In this sense, $K_{m, n}$ is the coherent union of the $K_i.$ As such, the density of $m$ and $n$ is readily calculated using linear algebra, and general density questions should also be straightforward. There appears to be no analog at all for the odd $m+n$ case.

There are many open problems associated with the Kolakoski sequence. Perhaps the most famous conjecture is that the limiting density of ``$1$'' in the Kolakoski sequence equals one-half. In this paper, ``density'' refers to the \textit{asymptotic} density \cite{OEIS} of a certain set of indices as a subset of $\mathbb{Z}_+,$ in this case, the indices of terms which are $1.$ There are various density-type conjectures that one can formulate about the Kolakoski sequence. The density-type conjectures dealing with fixed-distance observations can be naturally generalized into one ``generalized uniformness conjecture,'' (GUC) which we formulate. In order to formulate the GUC, we need to define functions $E_{m, n}, C_{m, n}$ which are naturally encountered when studying iterated run-length encoding or expansion as in the Kolakoski sequence. We will be using the same definitions and conventions as in Shen\cite{Shen}.

The GUC is far out of reach, but we prove a partial result. For $m, n, k > 0$ with $m+n$ odd, we define $G_{m, n, k}$ to be a directed graph (or multigraph for $k=1$). Its vertex set is $\{m, n\}^k.$ It has directed edges from $t$ to $C_{m, n}(m, t)$ and from $t$ to $C_{m, n}(n, t)$  for all $t$ in $\{m, n\}^k.$ The GUC easily implies that the graph $G_{m, n, k}$ is strongly connected, meaning that there are directed paths between all ordered pairs of vertices. We prove that these graphs are indeed connected. 
We think that this result is conceptually important because if we replace the Kolakoski sequence with a ``random sequence expanded $k$ times,'' then the $GUC$ is satisfied in expectation. In particular, in the directed graphs $G_{m, n, k}$, all vertices have in-degree and out-degree equal to $2,$ so the nullspace of the directed graph laplacian is given by the uniform-weight vectors.

The rest of this paper discusses ``correlations'' between terms of the sequence $K(m, n)$ which are $d$ terms apart. For $d>0,$ we define $cf(m,n,d)$ to be the density of indices $i$ such that $K(m, n)_i = K(m, n)_{i+d},$ if it exists. Naturally, all of these quantities are unknown. However, under the assumption of the GUC, these quantities can be computed. We define $tf(m, n, d)$ to be the theoretical value of $cf(m,n,d)$ assuming the GUC.

 Here is an alternative interpretation. We can define stochastic processes $P_{m, n}$ for $m+n$ odd. Each process $P_{m,n}$ is a collection of random variables $\{X_t\}_{t \in \mathbb{Z}}$ with values in $\{m,n\},$ and it is stationary, meaning that for all $d>0, k, k' \in \mathbb{Z}$ the marginal joint distribution of $(X_k,X_{k+1}\ldots, X_{k+d})$ is exactly equal to the marginal joint distribution of $(X_{k'}, X_{k' + 1}, \ldots, X_{k' + d}).$ We specify $P_{m,n}$ by specifying its marginal joint distributions on all finite consecutive sets; specifically, the probability that 
 \[(X_k, X_{k+1}, \ldots, X_{k+d}) = (a_0, \ldots, a_d) \in \{m,n\}^{d+1}\]
is equal to the asymptotic density of indices $i \in \mathbb{Z}_+$ such that $K(m,n)_{i + j} = a_j$ for all $0 \le j \le d$ assuming the GUC. In other words, this is the asymptotic frequency of $(a_0, \ldots, a_d)$ as a consecutive subsequence of $K(m,n).$ For example, assuming the GUC, the sequence $1,1,2,2$ has asymptotic frequency $1/18$ in $K(1,2).$ Consequently, for all $k,$ the stochastic process $P_{1,2}$ satisfies
\[ Pr[(X_k, \ldots, X_{k+3}) = (1,1,2,2)] = 1/18.\]
Under this interpretation, $tf(m,n,d)$ is the series of autocorrelations of the stationary stochastic process $P_{m,n}$. 

We will explain how to compute these frequencies assuming the GUC. Note that the stochastic processes $P_{m,n}$ exist without assuming the GUC. The statement that depends on the GUC is that these stochastic processes reflect the generalized Kolakoski sequences. 

We compute our algorithms for computing many exact and estimated values of $tf(m,n,d).$ We show many plots of these estimated values for $d \le 10^5$ and note many compelling patterns. We think that the series $tf(m,n,d)$ exhibit patterns that are very unusual for autocorrelation functions. 


This paper uses the same conventions as Shen\cite{Shen} as well as the definitions of $E_{m, n}$ and $C_{m, n}.$ 

\subsection{Outline}
In section 2, we formulate the GUC and prove that certain related directed graphs are strongly connected. In section 3, we prove a theorem that allows one to compute the frequency of any finite sequence as a consecutive subsequence of $K(m, n)$, assuming the GUC. In turn, this allows one to compute $cf(m, n, d)$ assuming the GUC. In section 4, we describe our algorithms for somewhat efficiently computing $cf(m,n,d)$, especially for $m+n \le 7.$  In section 5, we provide many plots of computed values of $cf(m, n, d)$ and discuss many striking patterns.

\section{The Generalized Uniformness Conjecture}\label{sec-guc}
This section makes extensive use of the functions $E_{m, n}$ and $C_{m, n}$. Both of these are functions which map pairs of sequences to pairs of sequences. We also use the fact that $C_{m, n}(s, t)$ always has the same length as $t.$ We recommend reading section 2 of Shen\cite{Shen}.

We first discuss another way to write $K(m, n)$ which naturally follows from the fact that $K(m, n) = E(K(m, n), m).$ By induction, we have for all $k>0$ that $ K(m, n) = E(K(m, n), m^k).$ In turn, 
\[ K(m, n) = E(K(m, n)_1 K(m, n)_2 K(m, n)_3 \ldots, m^k),\] where we tautologically regard $K(m, n)$ on the right hand side as an infinite concatenation of length-one sequences.

By identity $(1)$ in \cite{Shen}, we can expand the right hand side as 
\[K(m, n) =  E(K(m, n), m^k) = E(K(m, n)_1, t^{(1)}) E(K(m, n)_2, t^{(2)})E(K(m, n)_3 t^{(3)}) \cdots, \]
where we inductively define $t^{(1)} = m^k$ and for $i>0,$ $t^{(i+1)} = C(K(m, n)_i, t^{(i)}).$ Since $C_{m, n}(s, t)$ has the same length as $t,$ all of the sequences $t^{(i)}$ have length $k,$ so there are only $2^{k+1}$ possible values. In addition, $K(m, n)_i$ is a sequence of length one, so there are a total of $2^{k+1}$ possible values of the pair $(K(m, n)_i, t^{(i)}).$ In this sense, we have expressed the infinite sequence $K(m, n)$ as a concatenation of $2^{k+1}$ types atomic sequences. Perhaps a natural conjecture to ask is the following: the generalized uniformness conjecture.

\begin{conjecture}[Generalized Uniformness Conjecture (GUC)] \label{uniform}  Let $m, n>0, k \ge 0$ with $m+n$ odd. Define the sequences $t^{(i)}$ for $i \ge 1$ recursively by $t^{(1)} = m^k$ and $t^{(i+1)} = C(K(m, n)_i, t^{(i)}).$ 

 Let $x$ be in $\{m, n\}$, and let $t$ be in $\{m, n\}^k.$ Let $S$ be the subset of $\mathbb{N}$ of all $i$ such that $K(m, n)_i = x$ and $t^{(i)} = t.$ Then the \textit{asymptotic density} of $S$ equals $2^{-k-1}.$  
 \end{conjecture}

For $k=0, m = 1, n = 2$ this conjecture reduces to the classical conjecture that the limiting density of $1$ in the Kolakoski sequence equals $1/2.$

We now present a rather small partial result for the GUC. Recall that  for $i>0,$ $t^{(i+1)} = C(K(m, n)_i, t^{(i)}).$ There are only two possible values for $t^{(i+1)}$ given $t^{(i)}.$ It is natural to define a directed graph whose edges are the two possible choices. This is formalized in the following definition.

\begin{definition} For $m, n, k > 0$ with $m+n$ odd, we define $G_{m, n, k}$ to be a directed graph (multigraph if $k = 1$). Its vertex set is $\{m, n\}^k.$ For all $t$ in $\{m, n\}^k,$ the graph has an edge from $t$ to $C_{m, n}(m, t)$ and $C_{m, n}(n, t),$ and these are all of the edges. \end{definition}

For fixed $m, n, k,$ the sequence $t^{(1)}, t^{(2)}, \ldots $ is an infinite walk on the directed graph $G_{m, n, k}.$ The GUC implies that the density of any fixed $t$ in $\{m, n\}^k$ among this infintie sequence is $2^{-k}.$ In particular, this implies that $G_{m, n, k}$ is strongly connected, meaning that there are directed paths from any starting vertex to any ending vertex. We now prove this fact unconditionally.

\begin{theorem} \label{connected} The directed graph $G_{m, n, k}$ is strongly connected. \end{theorem}
\begin{proof}
Fix $m,n$ with $m+n$ odd. We now suppress the subscripts on $C_{m, n}$ and $E_{m, n}.$ We proceed by induction on $k.$ The base case, $k = 1,$ is trivial. $G_{m, n, k}$ is a multigraph with two vertices and two edges in each direction..

Now assume that $k \ge 1$ and $G_{m, n, k}$ is strongly connected. To show that $G_{m, n, k+1},$ it suffices to show that for fixed $t, u$ in $\{m, n\}^{k+1},$ there is a directed path from $t$ to $u.$ Let $t'$ be the sequence $t$ without its first term, and define $u'$ similarly. Then $t = t_1 t',$ and $u = u_1 u'.$

We are looking for a directed path in $G_{m, n, k+1}$ from $t$ to $u.$ A directed path is the same as applying either $C(m, -)$ or $C(n, -)$ repeatedly. By identity $(3)$, this is equivalent to finding a finite sequence $p$ with values in $\{m, n\}$ such that $u = C(p, t),$ where $p$ gives exactly the sequence of functions $C(m/n, -)$ applied.

By identity $(4)$ in Shen\cite{Shen}, the statement $u = C(s, t)$ is equivalent to the following
\begin{align*} u_1 u' &= C(s, t_1 t') \\ u_1 u'&= C(s, t_1) C(E(s, t_1), t') \\ u_1 &= C(s, t_1) \wedge u' = C(E(s, t_1), t'). \end{align*}

Instead, suppose that we were looking for the sequence $E(s, t_1).$ More formally, suppose that $r$ is a finite sequence with values in $\{m, n\}$ such that $R(r)$ (the run-length encoding of $r$) has values in $\{m, n\},$ $r_1 = t_1,$ $r_{|r|} \neq u_1,$ and $u' = C(r, t')$. Then defining $s:= R(r),$ we have $r = E(s, t_1),$ $u' = C(E(s, t_1), t'),$ and $C(s, t_1)$ is the complement of the last term of $E(s, t_1) = r,$ or $u_1$ by construction. We now construct $r.$

The condition that $R(r)$ has values in $\{m, n\}$ is the hardest, so we ignore that for now. Given $u', t'$ in $V(G_{m, n, k}),$ there exists a vertex $v$ such that $u' = C(m+n-u_1, v).$ (Note that $m+n-u_1$ is the complement of $u_1$.) This is because by Proposition 3.1 of Shen\cite{Shen}, the functions $C(m, -)$ and $C(n, -)$ are length-preserving bijections. On the other end, $C(t_1, t')$ is a vertex in $G_{m, n, k}.$ By the inductive hypothesis, there exists a directed graph from $C(t_1, t')$ to $v.$ In other words, there is a directed path from $t'$ to $u'$ such that the resulting sequence of ``labels,'' $p$, satisfies $u' = C(p, t')$, $p_1 = t_1,$ and $p_{|p|} \neq u_1.$ The sequence $p$ satisfies all of the conditions that we want $r$ to satisfy except that we require $R(r)$ to have values in $\{m, n\}.$ (The sequence of run-lengths of $p$ is arbitrary as far as we know.) We now adjust $p$ carefully.

WLOG, $n > m.$ (This is because $G_{m, n, k}$ is symmetric with respect to $m, n$.) In Shen\cite{Shen}, we also show that for any finite sequence $s$, $C(s, -)$ is a bijection on $\{m, n\}^{k}$; moreover, all of the orbit lengths are powers of two. Therefore, raising $C(s, -)$ to the power of $2^k$ yields the identity map on $\{m, n\}^k.$ Let $K:= 2^k.$ Then $C(s^K, -)$ is a bijection on $\{m, n\}^k.$ Observe that $K$ is a power of $2,$ and $n-m$ is odd, so $\gcd(K, n-m) = 1.$ Therefore, we can insert finitely many copies of $m^K$ or $n^K$ in between consecutive terms of $p$ to yield a sequence $p'$ such that all run-lengths in $p'$ are multiples of $n-m.$ We do not indernt $m^K$ or $n^K$ at the beginning or end. For example, if $n-m = 3$, $K = 4,$ and $p = n, m, m, m, m, n, n,$ one choice is $p' = n \ n^K \ n^K \ mmmm \ m^K \ m^K \ n^K \ nn.$ 

It turns out that we need two cases.

 \textit{Case 1:} $t_1 = u_1.$ Then the first and last terms of $p$ are not equal, and likewise for $p'.$ Let $q = p'.$ Let $\ell$ be the first term of $q.$ Partition $q$ into consecutive blocks of length $n-m.$ By construction, each block is either all $m$ or all $n.$ Define $r$ to be the sequence formed by inserting, between any two consecutive blocks (but not the beginning or end), a copy of 
 \[ \left(\ell^{m}(m+n-\ell)^{m} \right)^K. \] 
 
 We claim that $r$ satisfies the desired conditions. We still have $r_1 = t_1, r_{|r|} \neq u_1$ because we have not added to the beginning nor the end of $p$ to form $r.$ We still have that $r$ has values in $\{m, n\}.$ We still have $C(r, u') = t'$ because originally, $C(p, u') = t',$ and we have only inserted blocks that are the identity permutation on $\{m, n\}^k.$ Finally, consider the sequence of run lengths of $r.$ The first and last run lengths are $n$ because the first term of $q$ is $\ell,$ and the last term of $q$ is not $\ell.$ Most of the other run lengths are $m.$ The possible deviations are among the interior blocks of $q.$ However, each interior block of $q$ forms a complete run with either the $\ell^m$ succeeding it or the $(m+n-\ell)^m$ preceding it, for a total length of $n.$ Therefore, $R(r)$ has values in $\{m, n\},$ as desired.
 
 Here is an example. Suppose that $n = 3, m = 2, K = 2, q = 2,2,3,2,3.$  We insert $(2^2 3^2)^2$ between consecutive ``blocks'' of q (the blocks have length $1$), and we get 
 \[r = 2 \ 22332233 \ 2 \ 22332233 \ 3 \ 22332233 \ 2 \ 22332233 \ 3.\]

\textit{Case 2:} $t_1 \neq u_1.$ The first and last terms of $p$ are equal. From the example, it is clear that we need a slightly different adjustment.

Let $\ell$ be the last term of $p,$ which equals the last term of $p',$ and form the sequence $q$ by inserting a few more copies of $\ell^K$ before the last term of $p'$ so that the last run length of $q$ is congruent to $m \pmod{n-m}$ and also at least $m.$ Partition $q$ into blocks such that the last block has length $m,$ and all other blocks have length $n-m.$ Again, each block is either all $m$ or all $n.$ Define $r$ to be the sequence formed by inserting, betwee, any two consecutive blocks, a copy of 
 \[ \left(\ell^{m}(m+n-\ell)^{m} \right)^K. \] 
 It is easy to see that $r$ has the desired properties. 
 
 Here is an example. Suppose that $n = 3, m = 2, K = 2, q = 2,2,3,2,2,2.$ We get
 \[r = 2 \ 22332233 \ 2 \ 22332233 \ 3 \ 22332233 \ 2 \ 22332233 \ 22.\]

In either case, the sequence $s: = r$ satisfies $C(s, t) = u.$
\end{proof}

\section{Computing frequencies of subsequences in $K(m,n)$ assuming the GUC}
Let $m, n, d$ be positive integers with $m \neq n.$ We define $cf(m,n,d)$, or the correlation frequency in $K(m, n)$ at distance $d$, to be the asymptotic density of the set of indices $i$ such that $K(m, n)_i = K(m, n)_{i+d},$ if this limit exists.  If $cf(m, n, d)$ exists, it lies in $[0, 1]$.

Given that when $m+n$ is even, $K(m, n)$ is described by substitution rules, we expect $cf(m, n, d)$ for $m+n$ to be less interesting and possibly nonexistent. Therefore, we restrict our attention to the case when $m+n$ is odd. When $m+n$ is odd, the author's intuition is that $K(m, n)$ is chaotic and that $cf(m, n, d)$ should exist. A priori, this correlation frequency function is not particularly interesting: we expect it to decay to $1/2$ somewhat but not too regularly. However, empirical computations suggest that this function is quite interesting.

The values $cf(m,n,d)$ are unknown, but they can be calculate assuming Conjecture \ref{uniform}. As an intermediate step, we calculate the density of all possible finite consecutive subsequences. For example, the density of $m,n,n$ as a consecutive subsequence in $K(m,n)$ is by definition the asymptotic density of the set $\{i \in \mathbb{Z}_+ : K(m,n)_i = m, K(m,n)_{i+1} = n, K(m,n)_{i+2}=n\}.$

\begin{proposition}\label{density-recursion} Let $m \neq n.$ Assume Conjecture \ref{uniform}. Let $s$ be a finite nonempty sequence with values in $\{m,n\}.$ Let $t = R(s),$ the sequence of run lengths of $s.$ Assume that $|t|\ge 2.$ (If $|t|=1,$ then this proposition doesn't apply.)

If the first or last values of $t$ are greater than $\max(m,n),$ then the density of $s$ in $K(m,n)$ is $0.$ If some term of $t$ besides the first or last is not in $\{m,n\},$ then the density of $s$ is $0.$ 

If neither of these is the case, then construct the sequence $u$ as follows. Start with $t.$ If the first value is at most $\min(m,n)$ then remove it. Otherwise, replace the first value with $\max(m,n).$ If the last value is at most $\min(m,n)$ then remove it. Otherwise, replace the last value with $\max(m,n).$ Let $u$ be the resulting sequence, which is possibly the empty sequence. Then the density of $s$ equals the density of $u$ times $(m+n)^{-1},$ where the density of the empty sequence is $1.$
\end{proposition}

\begin{proof}[Proof. (Sketch)] The second paragraph is straightforward: recall that $R(K(m,n)) = K(m,n).$ If $s$ is a subsequence of $K(m,n),$ then we know for sure that $C(s)$ excluding the first and last terms of $C(s)$ is a subsequence of $K(m,n).$ Moreover, there exists some sequence $v$ which is almost equal to $C(s)$ except the first and last terms of $v$ may be greater than the respective terms in $C(s)$ such that $v$ is a subsequence of $K(m,n).$ 

Now assume that $t$ does not satisfy the hypotheses in the second paragraph. Using the identity $K(m,n) = E(K(m,n), m),$ we have a natural bijection between terms of $K(m,n)$ on the right hand side and full runs of $K(m,n)$ on the left hand side. For example, with $K:= K(1,2) = 1,2,2,1,1,2,1,2,2,1,2, \ldots,$ the term $K_6 = 2$ corresponds to the full run $K_8,K_9.$ The term $K_7=1$ corresponds to the full run $K_{10}.$ We also have a bijection between subsequences on the right and subsequences of full runs on the left. Note that $K_7,K_8$ is a subsequence of runs, but not a subsequence of full runs. 

We can't apply this directly to $s$ since $s$ is not a subsequence of full runs: it may have partial runs on the beginning/end. However, almost every occurrence of $s$ in $K(m,n)$ can be associated bijectively with an occurrence of $s',$ where $s'$ is a sequence of full runs. We form $s'$ as follows. Start with the sequence $s.$ If the first run of $s$ has length at most $\min(m,n),$ then remove it. Otherwise, replace the first run with a run of length $\max(m,n)$ of the same value. If the last run of $s$ has length at most $\min(m,n),$ then remove it. Otherwise, replace the last run with a run of length $\max(m,n)$ of the same value. Let $s'$ be the resulting sequence. Note that by construction, $R(s') = u.$

We claim that there is an almost-bijection between occurrences of $s$ in $K(m,n)$ and occurrences of $s'$ in $K(m,n)$ in which we require $s'$ to be a sequence of full runs. (There is possibly one unpaired occurrence. This discrepancy disappears in the limit.) The idea is that if $v$ is any sequence of full runs in $K(m,n),$ such as $1,2,2,1,2,2$ in $K(1,2),$ then we know that usually, $v$ is bookended by runs of length $\min(m,n)$ after "flipping the value" appropriate value, so $1,2,2,1,2,2$ would become $2,1,2,2,1,2,2,1.$ The possible exception is when $v$ is at the beginning of $K(m,n).$ Note that $v$ cannot be contained in $1,1,2,2,1,2,2,1$ because by assumption, $v$ is a sequence of full runs. This step uses the assumption that $|t| \ge 2.$ We omit the details.

Next, observe that there is a bijection between occurrences of $s'$ or the complement of $s'$ in $K(m,n)$, both required to be subsequences of full runs, and occurrences of $u$ as a subsequence in $K(m,n).$ Using Conjecture \ref{uniform}, one can show that the density of $s'$ and the density of the complement of $s',$ both as subsequences of full runs, are equal. Lastly, we must take into account that index $i$ in $K(m,n)$ "on the right" does not correspond to index $i$ in $K(m,n)$ "on the left," but rather to index $ i (m+n)/2 + o(i)$ as $i \rightarrow \infty.$
\end{proof}


We now have a prescription for computing $cf(m,n,d)$ assuming Conjecture \ref{uniform}. To avoid confusion, we will define $tf(m, n, d)$ to denote the theoretical correlation frequency assuming Conjecture \ref{uniform}. We define the \textit{uniform frequency} of a finite sequence in $K(m, n)$ to be the density assuming Conjecture \ref{uniform}. As opposed to frequencies, uniform frequencies can be calculated exactly according to Proposition \ref{density-recursion} and with the appropriate base cases. 

\begin{definition} For $m, n, d > 0$ with $m+n$ odd, $tf(m,n,d)$ is the correlation frequency of terms with indices $d$ apart in $K(m,n)$ assuming Conjecture \ref{uniform}. For finite sequences $u,$ the \textit{uniform frequency} of $u$ is its frequency as a consecutive subsequence in $K(m,n)$ assuming Conjecture \ref{uniform}. \end{definition}

We compute the uniform density of finite sequences $s$ with values in $\{m,n\}$ recursively based on the length of $s.$ Our base cases are $|s| = 0$ and $|R(s)| = 1$. To deal with this case $|R(s)| = 1$, we observe that by Conjecture \ref{uniform}, $K(m,n)$ is composed of runs, of which one fourth are each $m^m, m^n, n^m,$ and $n^n.$ For the purposes of computing the density of $s$ with $|R(s)| = 1,$ we may assume that $K(m,n)$ is $m^mn^mm^nn^n$ repeated. Our recursive step is Proposition \ref{density-recursion}. Finally, we define $tf(m, n, d)$ to be the sum over all length $d+1$ sequences with equal first and last term of their uniform frequencies.

Propositon \ref{density-recursion} provides an efficient way to compute the uniform frequency of one sequence, but we now present a more efficient way to compute $tf(m,n,d)$ by describing an infinite, periodic sequence which approximates the Kolakoski sequence. Basically, for each $(m,n),$ we will construct a family of periodic sequences $S^{(i)}.$ We will show that there exists an exponentially growing sequence of integers $D_i$ such that for all $i$ and all sequences $u$ of length less than $D_i,$ the frequency of $u$ in $S^{(i)}$ equals the uniform frequency of $u.$ First, we will define a sequence $D_i.$ We do not claim that these $D_i$ are tight bounds.

\begin{definition}\label{di}Assume that $m,n$ are fixed with $m<n.$ (This loses no generality since $tf(m,n,d) = tf(n,m,d).$) We define the sequence $(D_i)_{i \in \mathbb{Z}_+}$ as follows. For each $i >0,$ $D_i$ is defined as the shortest length of a sequence which can be obtained in the following process.

Pick $t^{(1)}, t^{(2)},  \ldots  t^{(i)} \in \{m,n\}$ arbitrarily. Let $s^{(0)}$ be the empty sequence. For $j \ge 0,$ define $s^{(j+1)} = E(s^{(j)} m, t^{(j+1)}),$ where the first argument has $m$ appended to $s^{(j)}.$ The term $D_i$ is defined to be the minimum possible length of $s^{(i)}$ over the finite set of choices.
\end{definition}

\begin{remark}
For any single set of choices, $|s^{(j+1)}| \ge m + m |s^{(j)}|.$ Therefore, $D_i > m^i.$ One can also show an exponential lower bound on $D_i$ if $m = 1,$ which we omit. Intuitively, as $i$ becomes large, expanding a sequence $s$ by a single letter $t$ multiplies its length by approximately $(m+n)/2$ rather than $m,$ provided that the sequence $s$ itself is the result of several iterated expansions. (If $s$ is arbitrary, then $E(mmmm \ldots, m)$ provides a counterexample.) Of course, knowing this for sure is essentially as hard as proving that the density of $m$ in $K(m,n)$ equals $1/2,$ which is to say, very hard. However, heuristically, we expect $D^i$ to grow nearly and possibly as fast as $((m+n)/2)^i.$
\end{remark}

\begin{lemma} \label{first-terms} Let $m<n$ with $m+n$ odd. Fix $t$ be in $\{m,n\}^k.$ Let $y$ be an infinite sequence with values in $\{m,n\}.$ Define $D_i$ as in Definition \ref{di}. Then the first $D_i$ terms of $E(y,t)$ are independent of $y.$
\end{lemma}

\begin{proof} For $1 \le j \le i,$ let $t^{(j)} = t_j$. Define $s^{(0)}, \ldots, s^{(i)}$ as in Definition \ref{di}. We prove by induction on $j$ that for $0 \le j \le i,$ the first $|s^{(j)}|$ terms of infinite sequence $E(y, t_1t_2\cdots t_j)$ agree with $s^{(j)}.$ 

The base case, $j = 0,$ is vacuously true. Suppose that this is true for $j.$ We have 
\[E(y, t_1 t_2 \cdots t_j t_{j+1}) = E( E(y, t_1 t_2 \cdots t_j), t_{j+1}).\]
By the inductive hypothesis, the first $|s^{(j)}$ terms of $E(y, t_1 t_2 \cdots t_j)$ agree with $s^{(j)}.$  Therefore, the infinite sequence $E(y, t_1 t_2 \cdots t_j)$ has the form $s^{(j)} a_1 a_2 a_3 \cdots,$ where $a_1, \ldots$ are arbitrary elements of $\{m,n\}$ The expansion by $t_{j+1}$ has the form
\[ E(s^{(j)}, t_{j+1}) E(a_1, C(s^{(j)}, t_{j+1})) \cdots.\]
In particular, no matter what $a_i$ is, the sequence $ E(a_1, C(s^{(j)}, t_{j+1}))$ begins with $m$ copies of the term $ C(s^{(j)}, t_{j+1})$ (since $m<n$). Equivalently, the initial segment of the sequence $E(y, t_1 t_2 \cdots t_{j+1})$ agrees with $E(s^{(j)} m, t_{j+1}),$ which equals $s^{(j+1)},$ as desired.
\end{proof}

\begin{proposition}\label{periodic}Let $m \neq n.$ $k >0,$ and $s$ be a finite sequence. Define the sequences $t^{(i)}$ such that $t^{(1)}$ is in $\{m,n\}^k$ and for $1 \le i \le |s|,$ $t^{(i+1)} = C(s_i, t^{(i)}).$ Assume that $t^{(|s|+1)} = t^{(1)}$. Also assume that for all $x$ in $\{m,n\}$ and $t$ in $\{m,n\}^k,$ exactly $2^{-k-1}$ of the elements $i$ in $\{1, \ldots, s\}$ satisfy $s_i = x$ and $t^{(i)} = t.$ 

Then for all sequences $u$ with values in $\{m,n\}$ and length at most $D_i+1,$ the frequency of $u$ in the infinite sequence $E(s s s s \cdots, t^{(1)})$ equals the uniform frequency of $u,$ where $D_i$  is defined in Definition \ref{di}.

\end{proposition}

\begin{corollary}
Let $s, t^{(1)}$ satisfy the conditions of Proposition \ref{periodic}. Let $d < D_i - 1.$ Then $tf(m,n,d)$ equals the correlation of terms $d$ apart in $E(s s s \cdots, t^{(1)})$; equivalently, the correlation frequency of terms $d$ apart in the finite sequence $E(s,t)$ with indices taken cyclically.
\end{corollary}

\begin{proof}
For all $d,$ $tf(m,n,d)$ equals the sum over all $u$ in $\{m,n\}^{d+1}$ with equal first and last terms of the uniform frequency of $u.$ Let $d < D_i - 1.$ By Proposition \ref{periodic}, the 
uniform frequency of $u$ equals the frequency of $u$ in $E(s s s \cdots, t^{(1)}).$ Therefore, $tf(m,n,d)$ equals the correlation frequency of terms $d$ apart in $E(s s s \cdots, t^{(1)}).$
\end{proof}

\begin{proof}[Proof of Proposition \ref{periodic}] 
WLOG, $n > m.$ This is because $K(n,m)$ and $K(m,n)$ have the same uniform frequencies. By identity $(1)$ in Shen\cite{Shen}, we can expand $E(s, t^{(1)})$ as
\[ E(s_1, t^{(1)}) E(s_2, t^{(2)}) \cdots E_(s_{|s|}, t^{(|s|)}).\]
By assumption, $C(s_{|s|}, t^{(|s|)}) = t^{(1)}.$ Therefore, we can expand $E(s s s \cdots t^{(1)})$ as an infinite repetition of the above sequence.

Let $u$ be a finite sequence. Assuming the GUC, the uniform frequency of $u$ can be thought of as follows. First, we choose a ``uniformly random index'' of $K(m,n),$ described in the next paragraph. Then, we compare $u$ to $K(m,n)$ for the next $|u|$ terms. The uniform frequency of $u$ equals the probability that all $|u|$ elements agree. 

 Consider all $2^{k+1}$ terms of the form $E(x, v),$ where $x \in \{m, n\}, v \in \{m, n\}^k.$ These sequences have a total length $N_k$. (By induction, one can show $N_k = 2 \cdot (m+n)^{k+1}.$ ) We can now choose one random term of one of the sequences such that all terms in the $2^{k+1}$ sequences have a $N_k^{-1}$ probability of being chosen. Intuitively, this random process matches the process of choosing a random index according to the notion of asymptotic density. We omit the details. At this point, we have that the uniform frequency of $u$ equals
 \[ \sum_{x \in \{m, n\}, t \in \{m, n\}^k, j \in \{1, \ldots, |E(x,t)|\}} N_k^{-1} \text{Pr}[\text{match given (x,t,j)}].\]
 
  Suppose that we have chosen term $j$ of $E(x,v).$ This $E(x,v)$ represents a subsequence of $K(m,n).$ We want to compare the next $|u|$ terms of $K(m,n)$ to $u.$ If $j$ is an index that is not one of the last $|u|-1$ indices of $E(x,v),$ then the proposition that the next $|u|$ terms of $K(m,n)$ agrees with $u$ only depends on $E(x,v)$ and $j.$ 

However, $j$ might be one of the last $|u|-1$ indices of $E(x,v).$ In this case, the proposition depends on the terms of $K(m,n)$ after $E(x,v)$; up to $|u|-1$ terms to be precise. The rest of $K(m,n)$ after $E(x,v)$ looks like $E(y, C(x, v)),$ where $y$ is an infinite sequence with values in $\{m, n\}.$ We also have the identity
 \[E(x,v) E(y, C(x,v)) = E(xy, v).\]
A priori, it is unknown if $u$ agrees with $E(xy, v)$ (starting with term $j$ of the latter). The key observation is that $E(y, C(x,v))$ has a certain number of terms which are independent of $y.$ 

For example, if $m = 1, n = 2, k = 2, C(x,v) = 22,$ then we claim that $E(y, 22)$ has a certain number of terms which are independent of $y.$ Indeed, $E(y,2)$ must begin $2, \ldots,$ and $E((2, \ldots), 2)$ must begin $2, 2, 1.$ Also suppose that $|u| = 4$ and $j$ is the last index of $E(x,v).$ We want to know if the next $4$ terms of $K(m,n)$ agree with $u.$ We have just shown that the three terms after $E(x,v)$ are always $2, 2, 1.$ Therefore, we know this fact for sure. On the other hand, if $|u| = 5,$ then the ``next $5$ terms of $K(m,n)$'' would involve the first four terms of $E(y, 22).$ The first four terms could be either $(2,2,1,2)$ or $(2,2,1,1);$ indeed, these two cases both have a probability in $(0,1)$ of happening. 

By Lemma \ref{first-terms}, the first $D_i$ terms of $E(y, C(x,v))$ are independent of $y.$ The sequence $u$ has length at most $D_i + 1.$ Therefore, the proposition that the first $|u|$ terms of $K(m,n)$ agrees with $u$ only depends on $x, v, $ and $j.$  In other words, we have replaced the expression $ \text{Pr}[\text{match given (x,t,j)}]$ with a deterministic function of $(x,t,j)$ to $\{0,1\}.$ 

Importantly, this statement is false for long sequences $u.$ For long sequences $u,$ the event of being a match given $x, v, j$ depends also on the terms of $y,$ so the probability is not in $\{0,1\}.$



On the other hand, we can repeat this analysis for finding the frequency of $u$ in $E(s s s s \cdots, t^{(1)}).$ It is now more straightforward to choose a uniformly random index. We choose term $j$ of $E(s_i, t^{(i)})$ with probability $|E(s, t^{(1)})|^{-1}.$ 


 Thus the frequency of $u$ in $E(s s s s \cdots, t^{(1)})$ equals
\[ \sum_{1 \le i \le |s|, 1 \le j \le |E(s_i, t^{(i)})|} |E(s, t^{(1)})|^{-1} \text{Pr}[\text{match }\ldots].\]
By assumption, an arbitrary pair $(x,t)$ equals $(s_i, t^{(i)})$ exactly $2^{-k-1}$ of the time. Therefore, the frequency of $u$ equals
 \[ \sum_{x \in \{m, n\}, t \in \{m, n\}^k, j \in \{1, \ldots, |E(x,t)|\}} N_k^{-1} \text{Pr}[\text{match } \ldots].\]
 A priori, the probability of a match has horrible dependencies on the order of the pairs $(s_i, t^{(i)})$ because we can no longer treat them as random in any sense. Fortunately, for $u$ which are short enough, the event of being a match only depends on $x, t, j.$ Therefore, we arrive at the same summation as before, and the frequency of $u$ in $E(s s s s \cdots, t^{(1)})$ equals its uniform frequency.
\end{proof}
%
%
%

\subsection{Interpretation as the autocorrelation function of a stochastic process}
For $m, n>0$ with $m+n$ odd, we construct a stochastic process $P_{m, n}$ as follows. This stochastic process is an infinite collection of random variables $\{X_t\}_{t \in \mathbb{Z}}$ with values in $\{m, n\}$. The author is not familiar with probability distributions on infinite-dimensional objects, but marginal distributions of any finite collection are easier to grasp. We define $P_{m, n}$ by explaining how to draw a joint sample out of any finite consecutive subsequence or ``window''. Each process is ``stationary,'' meaning that for all $d>0, k, k' \in \mathbb{Z}$ the marginal joint distribution of $(X_k,X_{k+1}\ldots, X_{k+d})$ is exactly equal to the marginal joint distribution of $(X_{k'}, X_{k' + 1}, \ldots, X_{k' + d}).$ To draw a joint sample of $d+1$ consecutive $X_k,$ we output the sequence $(a_0, \ldots, a_d) \in \{m, n\}^{d+1}$ with probability equal to the uniform frequency of $(a_0, \ldots, a_d)$ in $K(m, n)$ assuming the GUC. For convenience, we define $Q_{m, n}$ to be the same process, except we replace output values of $X_k$ which are $n$ by $+1$ and $m$ by $-1.$

The autocorrelation function of a stationary stochastic process indexed by the integers is defined to be a function from $d \in \mathbb{Z}_+$ that gives the correlation of two elements which are $d$ apart. This is a map from $\mathbb{Z}_+ \rightarrow [-1, 1].$ In out specific example, the autorrelation of $P_{m, n}$ at distance $d$ is exactly $tf(m, n, d).$ Thus a plot of $tf(m, n, d)$ versus $d$ is a plot of the autocorrelation function.

\textbf{Examples.} If the stochastic process is periodic with period $p$ (meaning any finite sample is periodic with period $p$), then the autocorrelation function will also be periodic with period $d.$ If 
the stochastic process has fully independent $X_i$ (at least on all finite windows), then all autocorrelations are zero. If the stochastic process is a Markov process that stabilizes, then the autocorrelation function tends to exponentially decrease to zero, but not necessarily smoothly.

The plots in the last section show that the autocorrelation function is quite remarkable for the three pairs $(m, n) \in \{(1,2), (2,3), (3,4)\}.$ The author is not aware of the significance of such a remarkable-looking autocorrelation function, but it seems very unusual.

\section{Algorithmicaly computing and estimating $tf(m, n, d)$}
All algorithms were implemented in $C++.$

We computed the exact uniform frequencies for $K(1, 2)$ only. 
Our main method is to find a periodic sequence according to Proposition \ref{periodic} and efficiently compute correlations in this periodic sequence. In particular, we used $k = 18$ with a period of length $2 \cdot 3^{18}.$ Theoretically, we only have a guarantee that for $d$ less than about $1400,$ $tf(m,n,d)$ will equal the correlation within the periodic sequence, and for much larger $d,$ we will only have an approximation. We mitigated this by running two trials for $k = 18.$ 

For $k = 18,$ it is easy to see that we want a finite directed walk on $G_{m, n, k}$ which may repeat edges but must be closed and use all edges the same number of times. The most straightforward way to do this is to use an Eulerian cycle.  In Theorem \ref{connected}, we proved that the graph must be connected. Also, Proposition 3.1 of Shen \cite{Shen} implies that all vertices have in- and out-degree equal to $2.$ Therefore, an Eulerian cycle must exist. The time to find a cycle is dominated by the rest of the algorithm. We incorporated many random choices into our algorithm to see if the resulting estimations for $1400 \le d \le 10^5$ changed much. Suppose that this cycle is represented by the sequence $s$ and the starting point $t.$ We next expanded $E(s, t)$ into a \texttt{vector<unsigned int>}, which is straightforward. Note that the length is $2 \cdot 3^{18}.$ Let $N:=2 \cdot 3^{18}.$

Now, we want to compute correlations in $E(s, t)$ for all distances, or at least for $d \le 100000.$ With a naive implementation, we loop, for all $d,$ through the entire sequence $E(s, t).$ The total number of comparisons is $10^5 \cdot N\approx 7.7 \cdot 10^{13},$ which is too many.\footnote{The author admits to having quite limited computational resources} Instead, we use discrete fourier transforms to compute the coefficients. We follow the conventions in Sutherland\cite{Suth}, section 3.4.3. Suppose that $f, g$ are polynomials defined by
\begin{align}\label{fg}f(x) = \sum_{i=1}^{N} (2 E(s, t)_i - 3) x^{(i-1)}, g(x) = (2 E(s, t)_1 - 3) +  \sum_{i=1}^{N - 1} (2 E(s, t)_{N + 1 - i} - 3) x^{i}.\end{align}
For example, if $k = 1,$ then one choice for $E(s,t)$ is $1,2,1,1,2,2.$ Then 
\[ f(x) = -1 + x - x^2 - x^3 + x^4 + x^5, g(x) = -1 + x + x^2 - x^3 - x^4 + x^5.\]

The coefficients for $g(x)$ are those of $f,$ reversed, and shifted so that they have the constant coefficients are both $(2 E(s, t)_1 - 3).$
Now consider $f(x) g(x) \pmod{x^N-1}$ i.e. with no terms of degree more than $N-1.$ One can prove that the coefficient of $x^d$ in $f(x) g(x) \pmod{x^N - 1}$ equals the number of times $E(s,t)_i = E(s, t)_{(i+d) \pmod{N}}$ minus the number of times $E(s,t)_i \neq E(s, t)_{(i+d) \pmod{N}}$. It is then straightforward to compute the correlations in $E(s, t)_i$ from these.
Let $\omega$ be a primitive $N^{\text{th}}$ root of unity such that differences between powers are invertible. Recall that in the polynomial representation for discrete fourier transforms, the discrete fourier transform of a polynomial $f$ with degree less than $N$ with respect to a primitive $N^{\text{th}}$ root of unity, $\omega$, such that differences between powers are invertible is the data $f(1), f(\omega), \ldots, f(\omega^{N-1}).$

Then applying classical results of discrete fourier transforms, we find that
\[\text{DFT}_{\omega}(fg \pmod{x^N - 1})  = \text{DFT}_{\omega}(f) \cdot \text{DFT}_{\omega}(g),\]
where $\cdot$ denotes elementwise multiplication. Also, we have the identity 
\[ \text{DFT}_{\omega^{-1}} \text{DFT}_{\omega} = N \cdot id.\]
Furthermore, these results hold in any ring with the stated conditions. We used the finite field of (prime) order equal to $5N+1,$ and we use the primitive $N^{\text{th}}$ root of unity $32.$ We applied an FFT type algorithm to compute the discrete fourier transforms quickly, outlined in the following lemma. Theoretically, its runtime is around $O(N \log N)$ field operations (with no inverses required).
\begin{lemma}[FFT of vectors whose length is a multiple of 3]  Let $f$ be a polynomial of degree less than $3M.$ Let $\omega$ be a primite $3M^{\text{th}}$ root of unity such that differences between powers are invertible. Suppose that $g, h, i$ are polynomials of degree less than $M$ such that 
\[f(x) = g(x) + x^M h(x) + x^{2M} i(x).\]
 Let 
\begin{align*} f_1(x)&:= g(x) + h(x) + i(x)\\ f_2(x) &:= g(x \omega) + \omega^M h(x \omega) \omega^2M i(x \omega) \\ f_3(x) &:= g(x \omega^2) + \omega^{2M} h(x \omega^2) + \omega^{4M} i(x \omega^2). \end{align*}
Suppose that the discrete fourier transforms of $f_1, f_2, f_3$ are computed with respect to $\omega^3.$ Then
\begin{align*} f(\omega^{3j}) &= f_1(\omega^{3j}) \\ f(\omega^{3j+1}) &= f_2(\omega^{3j}) \\ f(\omega^{3j+2}) &= f_3(\omega^{3j}). \end{align*}
\end{lemma}
\begin{proof} This is just straightforward substitution. \end{proof}

This lemma also outlines a recursive method to compute the discrete fourier transforms of a length $N$ sequence with respect to $32$ in the field $\mathbb{F}_{5N+1}.$ Fortuitously, $5N+1 = 3874204891 < 2^{32},$ so field elements are stored as \texttt{unsigned int}s. We must often perform arithmetic $\text{mod } 5N+1.$ Of course, multiplying unsigned integer and modding by $5N+1$ does not work. Instead, we casted all arguments to \texttt{unsigned long long}s, or $64$-bit positive integers. Modding by $5N+1$ was done judiciously to avoid overflow in 64-bit integers because $5N+1$ is so close to $2^{32}.$ Note that modular arithmetic calls \texttt{unsigned long long}s, but the data is converted to \texttt{vector<unsigned int>} types to save memory. The FFT algorithm does not require inverses, and we only need $32^{-1}$ and $N^{-1}$ to invert $DFT_{32},$ so we precomputed these in SageMath. 

We also used the field $\mathbb{F}_{5N+1}$ for smaller $k$ because we still have roots of unity of order $2 \cdot 3^j$ for $j \le 18,$ also precomputed in SageMath.


We also used DFT mod $5N+1$ for different values of $(m,n).$ By Proposition \ref{periodic}, for different values of $(m,n),$ there are different periodic approximations depending on $k$ in which one period has the form $E_{m, n}(s, t).$ One can show that the period is exactly $2(m+n)^k.$ Thus we cannot compute correlations using DFT in $\mathbb{F}_{5N+1}$ directly. Instead, we choose a value of $k$ so that the period is somewhat less than $N.$ For example, with $(m,n) = (2,3),$ we have $M:= 2(2+3)^{12} = 488281250 < N.$ We pad $E_{2,3}(s,t)$ with zeros to make a length $N$ vector or degree $N-1$ polynomial $f,$ with reverse $g$. We replace $3$ in equation (\ref{fg}) above with $5$ so that the coefficients of $f$ and $g$ are $\pm 1.$ We denote the terms of $f$ and $g$, respectively,  by
\[f = a_0, a_1, \ldots, a_{M-1}, 0_M, 0_{M+1}, \ldots, 0_{N-1}.\]
\[g = a_0, 0_{N-1}, 0_{N-2}, \ldots, 0_{M}, a_{M_1}, \ldots, a_1.\]
Here, $0_j$ just means zero, and the indices are for bookkeeping. Note that $g$ reversed is $f.$ For $d < N - M,$ The coefficient of $x^d$ $f(x) g(x) \pmod{x^N-1}$ equals 
\[ \sum_{i=0}^{M-1-d} a_i a_{i+d}.\]
We want
\[ \sum_{i=0}^{M-1} a_i a_{Mod(i+d, M)},\]
so we are missing $\sum_{i=M-d}^{M-1} a_i a_{i+d-M}.$ To compute these sums for $d$ up to a limit $D,$ we can define a third polynomial $h(x)$ by
\[ h = a_0, a_1, \ldots, a_{D-1}, 0_D, \ldots, 0_{N-1}.\]
If $D < N,$ then the values $\sum_{i = M-d}^{M-1} a_i a_{i+d-M}$ for $d < D$ are read off the coefficients of $h(x) h(x) \pmod{x^N-1}$ in a straightforward way.

To be concrete, for $m+n = 5,$ we choose $k = 12, 2\times 5^{12} = 488281250, D = 10^6.$ For $m+n=7,$ we choose $k = 10, 2\times 7^{10} = 564950498, D = 10^6.$ In this way, we can estimate $tf(m, n, d)$ for $d < D.$ Again, We can run multiple trials to see how much estimates of $tf(m,n,d)$ differ as a function of $d$ to estimate the point up to which this algorithm gives near exact values.

\begin{remark} The finite field of prime order $3 \cdot 2^{30} + 1$ is also convenient using \texttt{unsigned int}s. \end{remark}

\begin{remark} If one attempts to choose a higher $k,$ then the main roadblock for the algorithm outlined here is memory, not time. We estimate that our program uses about 7 GB and 20 minutes for $(m,n) = 1, k = 18.$  By $k = 20,$ it the polynomial vectors would have size $2 \times 3^{20} \approx 7 \times 10^9,$ and to multiply these  with our FFT algorithm would need 64-bit integers. On the other hand, modular arithmetic with 64-bit integers and without specialized algorithms requires at least 96-bit or 128-bit integers, which could cause an unanticipated slowdown.  \end{remark}

\section{Plots and tables of computed values} 
The plots were prepared with  mathematica. In all plots, the $x$-coordinate is proportional to $d,$ with different ratios in different plots. The $x$-coordinates can be ignored. Except in the first two plots, some data points for small $d$ are not within the $y$-range. We consider the patterns for larger $d$ to be more interesting.

\subsection{Comments on $K(1,2)$}
 The top plot of Figure \ref{fig:tf12d} is a plot of untransformed values of $tf(1,2,d)$. It shows many mysterious patterns, the most obvious of which is that the sign is determined by the remainder $\text{mod }3$ for small $d.$ The value $d = 782$ is the smallest $d$ which breaks this pattern; we hhave $tf(1,2,782) = \frac{2392527}{3^{14}}  \approx 0.500218 > \frac{1}{2}.$

As previously mentioned, our estimated values for $tf(1,2,d)$ are not expected to be exact for $d \ge 1400.$ 
Figure \ref{fig:tf12d-est} plots $tf(1,2,d)$ for $d \le 30000$ with an arbitrary scaling factor. We differentiate the data points $d, tf(1,2,d)$ based on $Mod(d,3)$ because this plot, we well as many smaller plots, suggest that the series $tf(1,2,d)$ behaves much differently for $d$ in the three residue classes $\text{mod }$ 3, even for large $d.$ Figure \ref{fig:tf12d-est1} shows only data points with $d \equiv 0 \pmod{3}.$ The series resembles a random walk which reverts to the mean more often than usual. The ``time'' scale of reverting to the mean does not appear to change much as $d$ increases, although this could be an artifact of our estimation method. The true values of $tf(1,2,d)$ may show much different patterns.

Figure \ref{fig:tf12d-est2} compares values of $tf(1,2,d)$ for two different random trials.  (Recall that two trials are different because we our Eulerian Cycle algorithm uses randomness.) They agree exactly for $d \le 5100.$ This suggests that our estimated values are mostly accurate for $d \le 5100.$ We see there are significant differences by $d = 10000,$ so our values for $d \ge 10000$ should be considered very inaccurate. However, we believe that the features of non-stationarity in our plots of estimated $tf(1,2,d)$ reflect the true values of $tf(1,2,d).$

\subsection{Comments on $K(2,3)$}
Figure \ref{fig:tf23d-est} shows $tf(2,3,d),$ and points in different residue classes $\text{mod }5$ are differentiated. The ``blue'' and ``green'' classes are very similar, as are the ``red'' and ``purple'' classes. They are not exacty the same, as can be verified by checking a list of small values up to $d = 30.$ In a second trial, we found that the values were exactly the same up to $d \approx 1.2 \times 10^5,$ which suggests that the estimates are near exact up to this point. In Figure \ref{fig:tf23d-est2}, we isolate just the class $d \equiv 0 \pmod{5}.$ The class $0 \pmod{5}$ is the disjoint union of $5$ classes $\pmod{25}.$ The bottom plot suggests that these 5 classes behave slightly differently. If the 5 classes behaved randomly within the series for $d \equiv 0 \pmod{5},$ then we would not expect to see one color concentrate on one ``side'' of the series locally, but this is what we see. We did not observe a similar pattern for $tf(1,2,d)$ with $d$ restricted to, for example, $0, 3, 6 \pmod{9},$ possibly because we were unable to compute exact values above $d = 5000.$

There appear to be many values of $d$ which are ``nodes'' of all five series. We think this is suggestive of a plot of the real and imaginary parts of a complex function. Specifically, we hypothesize that there exists a mostly smooth function $f: \mathbb{Z}_+ \rightarrow \mathbb{C}$ such that 
\[tf(2, 3, d) \approx Re(f(d) \cdot cis(2 \pi i d / 5).\]
A more parsimonious hypothesis is that the function $f$ actually oscillates in one real subspace of $\mathbb{C},$ which would imply that the five sub-series of $tf(2,3,d)$ $\text{mod}$ $5$ would have approximately constant ratios. Figure \ref{tf12d-est} suggests that series do not have constant ratios. Therefore, we think it is worthwhile to consider a general complex function $f.$

To test this hypothesis, we can compute the series
\[h(d) := tf(2,3,d) - 2 \cos(2 \pi/5) tf(2,3,d+1) + tf(2,3,d+2).\] If $f$ changes relatively slowly and the above identity has relatively small errors, then the series $h$ will have relatively small magnitude. The series $h(d)$ for $d \le 100000$ is shown in Figure \ref{fig:tf23d-res1}. Indeed, the series $h(d)$ has amplitude approximately one-tenth that of $tf(2,3,d).$ Having established this, we may want to know if the errors look like Gaussian noise or themselves have structure. Restricting the series to one class $\text{mod}$ $5$ shows no clear pattern. However, restricting to one class $\text{mod}$ $25$ shows obvious wave-like patterns, as shown in Figure \ref{fig:tf23d-res2}, although the ratio of the noise amplitide to the wave amplitude is much larger in this figure than in the plots of the untransformed series $tf(2,3,d).$ 

Because the series $h$ appears to behave differently in the different residue classes $\text{mod}$ $25$ and because the noise ratio is higher, we plot exponential moving averages of the 25 series  
\[g_j(c):= h(25 c + j)\] for $j = 0, \ldots, 24.$ By definition, the exponential moving average with decay parameter $p \in [0,1]$  of a discrete time series $j(t)$ is another time series $m(t)$ such that for all $t,$ $m(t)$ is a mean of $j(t), j(t-1), j(t-2), \ldots$ with relative weights $1, p, p^2, \ldots$. We use $p = 0.99.$ The exponential moving averages of the 25 series $g_j$ are shown in Figure \ref{fig:tf23d-res3}. We see a convincing pattern of wave-like series with common nodes, although the pattern breaks down by $c = 3000.$ This corresponds to $d = 75000.$  We think that these wave-like patterns are independent of and possibly orthogonal to those in Figure \ref{fig:tf23d-est}. This is because the wave-like forms in Figure \ref{fig:tf23d-est} show big oscillations that don't cross the x-axis. Again, the series appear to have predictable phase. Using rather ad-hoc estimation methods\footnote{For example, we might take a cross section at $c \approx 1000.$}, we predict that 
\[ h(d) \approx Re(f_1(d) \cdot cis(2 \cdot 3 \pi i / 25))\]
for a relatively smooth function $f_1: \mathbb{Z}_+ \rightarrow \mathbb{C}.$ In particular, our predicted phase multiplier  or ``angular momentum'' is $6 \pi i / 25,$ whereas our predicted phase for the plain function $tf(2,3,d)$ was $2 \pi i / 5.$

We expect that the mod $5$ and mod $25$ patterns for $(m,n)=(2,3)$ have analogs for $(m,n)=(1,2)$ and $(3,4).$ However, computing $tf(1,2,d)$ requires much more memory than computing $tf(2,3,d),$ and we think the patterns are most suggestive in the $tf(2,3,d)$ data.

\subsection{Comments on $K(3,4)$}
Figure \ref{fig:tf34d-est} shows $tf(3,4,d),$ and points in different residue classes $\text{mod 7}$ are differentiated. A different random trial suggests that these estimates from the $k=10$ estimation are accurate up to $500000,$ although we don't show the whole range because the pattern would be too compressed. We see that the seven classes contain three pairs that are quite similar. We see that several values of $d$ appear to be simultaneous nodes of the seven classes. Again, we hypothesize that there exists a mostly smooth function $f: \mathbb{Z} \rightarrow \mathbb{C}$ such that
\[ tf(3,4,d) \approx Re( f(d) \cdot cis(2 \pi i d / 7)).\]

\subsection{Comments on $K(1,4),  K(1,6),$ and $K(2,5)$}
Figures  \ref{fig:tf14d-est},  \ref{fig:tf16d-est}, and  \ref{fig:tf25d-est} show plots of $tf(1,4,d),$ $tf(1,6,d),$ and $tf(2,5,d)$, respectively. We expect that our computed values are very accurate for about $d \le 10^5, 10^{5.5}, 10^{5.5},$ respectively.
At first glance, all three series resemble ``stochastic volatility'' processes, of stochastic processes of the form 
\[tf(m,n,d) = N(d) V(d),\] where $N(d)$ are independent normal random variables, and $V(d)$ is a relatively smooth function of $d.$ 
However, by meticulously filtering out by various moduli and zooming in, we find a few patterns. 

The first plot in Figure \ref{fig:tf14d-est2} shows the series $tf(1,4,d)$ restricted to $d \equiv 0 \pmod{5}$ and $0 \le d \le 15000.$ There is a clear wave-like pattern in the right half of the plot, but with significantly more noise than the comparable $tf(2,3,d)$ plots; (and, as seen before, this noise doesn't seem to vanish for larger $d$). This motivates us to take exponential moving averages. The second plot in Figure \ref{fig:tf14d-est2} shows the exponential moving average with parameter $p = 0.95,$ or a ``time constant'' of about 20. This constant was chosen informally to reflect the fact that the waves for $tf(2,3,d)$ appear to have shorter wavelength than those in 
\[h(d) := tf(2,3,d) - 2 \cos(2 \pi/5) tf(2,3,d+1) + tf(2,3,d+2).\]

The third plot shows all five classes together. The five classes resemble waves that do not have common nodes. Instead, the wavelengths seem to be comparable, but the phases seem to be asynchronized, possibly approaching even spacing around the phase space. Also notable is that the different classes do not obviously have the same shape, as can be seen by comparing the second and fourth plots of Figure \ref{fig:tf14d-est2}. These wave patterns extend at least until $d = 100000,$ as long as we take exponential moving averages. We did not find an approximate linear relation between the five series. Even the sum of the five series shows clear wave-like patterns.

The series $tf(1,6,d)$ does not show clear modular structure until we take $\text{mod}$ $21.$ Recall that in all previous instances, the modulus was $m+n,$ but now, it is $3(m+n).$
The first plot in Figure \ref{fig:tf16d-est2} shows the series $tf(1,6,d)$ restricted to $d \equiv 0 \pmod{7}, d \le 420000,$ and the three collors distinguish the three sub-classes $\text{mod}$ $21.$ It is clear that these three sub-classes are different and that $7$ is not a natural modulus for splitting the series $tf(1,6,d)$. One can check that $3$ is not either. The second plot shows $tf(1,6,d)$ for $210000 \le d \le 567000$ and $Mod(d, 21) \in \{0, 4, 6, 8, 12\}$ (which were arbitrarily chosen), after applying an exponential moving average with $p = 0.99$ for each class separately. The value $p=0.99$ is probably too low. In any case, we do not see a clear pattern.

When analyzing $tf(2,5,d),$ we are really scavenging for patterns. Restricting to one class $\text{mod}$ $7$ yields seven series whose moving exponential averages show wave-like patterns, but without the averages, there appears to be too much noise to see a wave-like pattern. Thus, a natural hypothesis is that restricting $tf(2,5,d)$ to one class modulo a multiple of $7$ may split the series into series with less noise. After restricting to $d \le 100000$ and $d \equiv 0 \pmod{7 N}$ for $N \in \{1, \ldots, 25\},$ we did not find a series that showed a wave-like pattern before taking exponential means. In Figure \ref{fig:tf25d-est2}, the first plot shows $tf(2,5,d)$ restricted to $35000 \le d \le 70000.$ The second plot shows the exponential moving average of the series in the first plot with $p = 0.95.$ The third plot shows a series like the second plot, but  with $d \equiv 3 \pmod{7}.$ Note that the amplitude in the second and third plots are about one-quarter  of the amplitude in the first plot, meaning that the first plot has a lot of noise.

\subsection{Conclusion}
Of all of the patterns that we have noted, we think two are most important. First, when $n = m+1,$ the series $td(m,n,d)$ behaves much differently as we restrict $d$ to  one residue class $\text{mod }m+n.$ Second, for $(m,n) = (2,3)$ and $(3,4)$ and less so for $(1,2),$ the $(m+n)$ different series resemble waves that have simultaneous nodes. More speculatively, the series $tf(2,3,d)$ shows convincing patterns mod $25.$ This suggests that the series $tf(2,3,d)$ could be built from the combination of structure $\text{mod}$ $5,$ substructure $\text{mod}$ $25,$ and so on. It's possible that these patterns will completely disintegrate beyond the range that we have estimated here. To know for sure, one would need a lot of memory to extend our calculations or a more memory-efficient approach. If these patterns are transitory, it may still be interesting to ask why they are so convincing for the range that we have estimated here.

\begin{table}[h]
	\centering
	\caption{A table of correlation frequencies for the sequence $K(1,2)$, assuming the GUC.}
	\begin{tabular}{cl}
		\hline
		d & tf(1, 2, d)\\
		\hline
		1 &  2/3 \\
		2 &  2/3 \\
		3 & 2/9 \\
		4 & 2/3 \\
		5 &  2/3 \\
		6 &  8/27 \\
		7 &16/27 \\
		8 &16/27 \\
		9 &2/9 \\
		10 & 50/81 \\
		11 &  50/81 \\
		12 & 20/81 \\
		13 &2/3 \\
		14 & 2/3 \\
		15 & 22/81 \\
		16 & 146/243 \\
		\hline
	\end{tabular}
	\label{tab:diffs}
\end{table}

\begin{figure}[htbp]
	\centering
    \includegraphics[width=\textwidth]{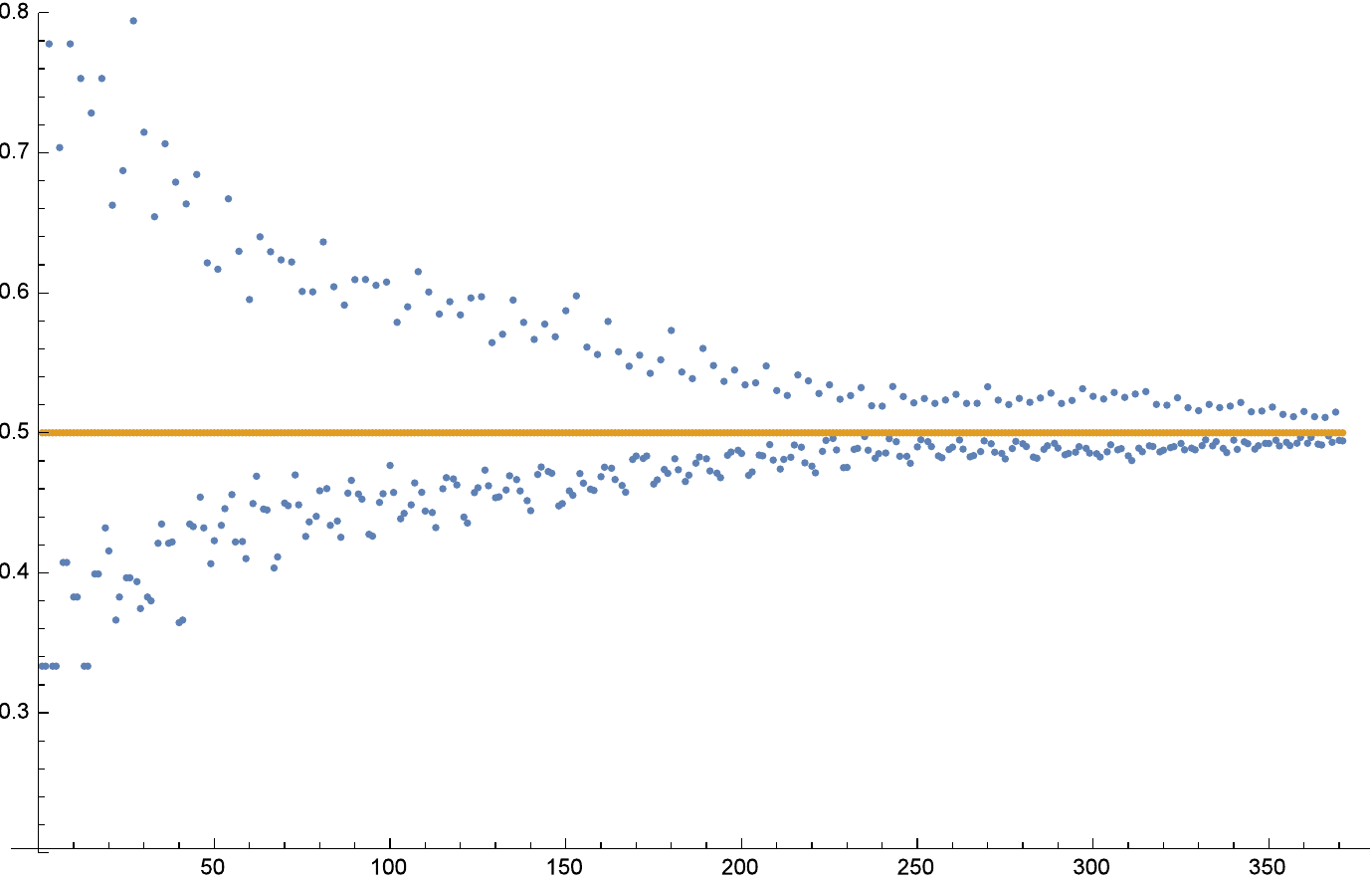}
       \includegraphics[width=\textwidth]{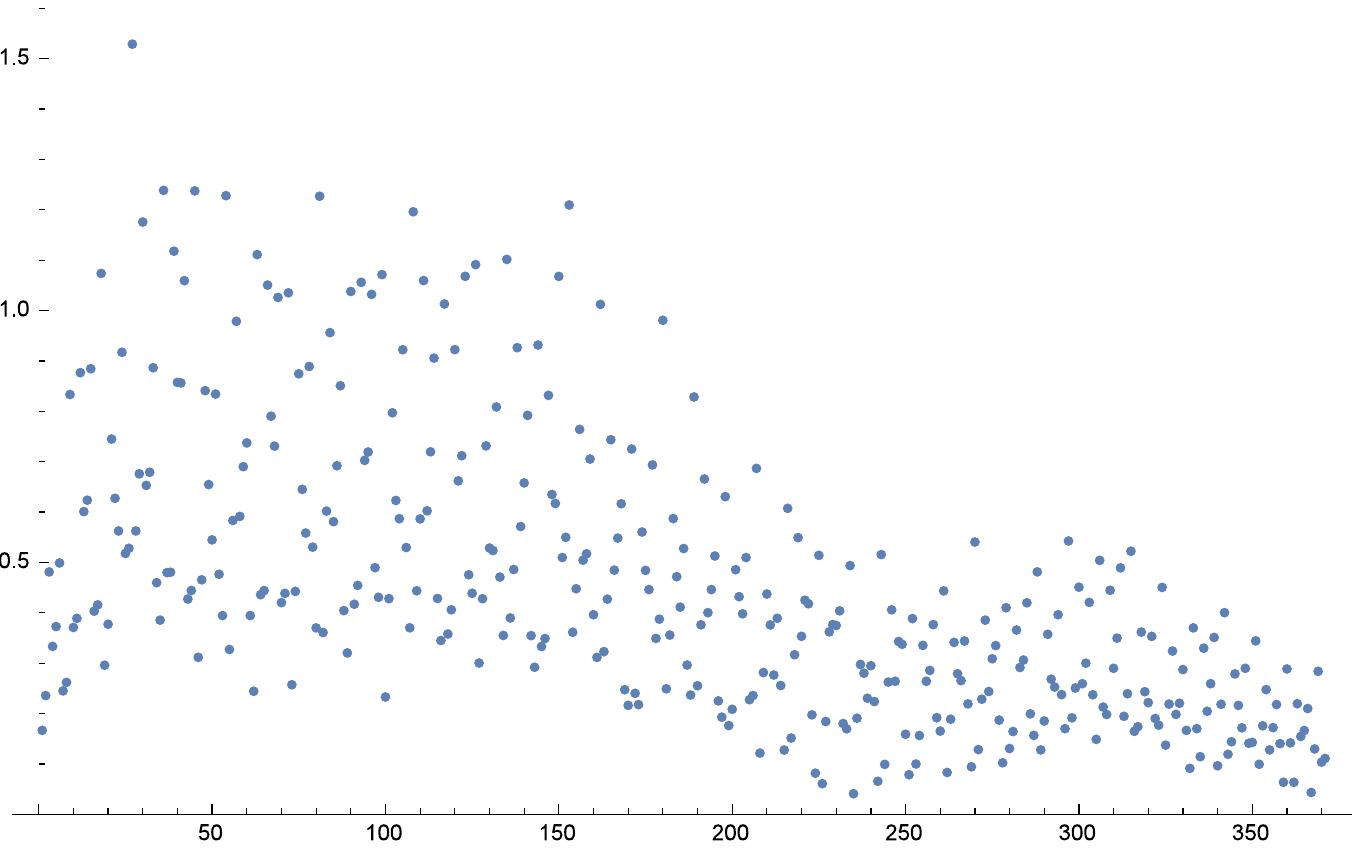}
    \caption{Top: A plot of exact $tf(1, 2, d)$ for $1\leq d\leq 400$. The dashed line is the line $y=1/2$. Bottom: A plot of exact $(tf(1,2,d) - 1/2) * \left(I(d \equiv 0 \pmod{3}) * 2 - 1\right) d^{1/2}$ for $1 \leq d \leq 400,$ or $tf(1,2,d)-1/2$ times the expected sign times an arbitrary amplifying factor. Both plots show all data points.}
    \label{fig:tf12d}
\end{figure}

\begin{figure}[htbp]
	\centering
    \includegraphics[width=\textwidth]{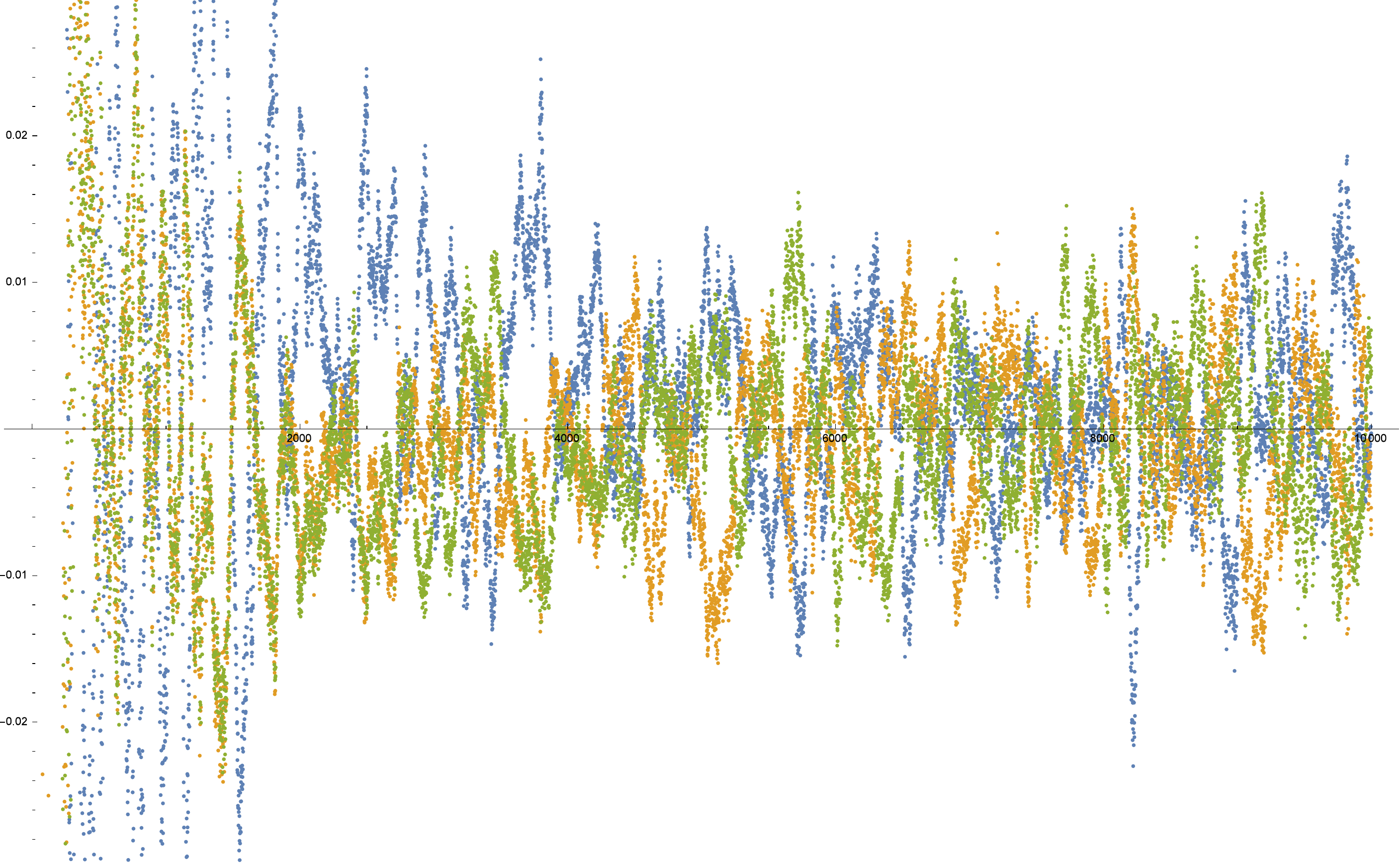}
    \caption{A plot of estimated $(tf(1,2,d) - 1/2) \cdot d^{1/2}$ for $d \leq 30000.$ Green, blue, and orange differentiate points where $d$ is $0, 1,$ and $2 \pmod{3},$ respectively.}
    \label{fig:tf12d-est}
\end{figure}

\begin{figure}[htbp]
	\centering
    \includegraphics[width=\textwidth]{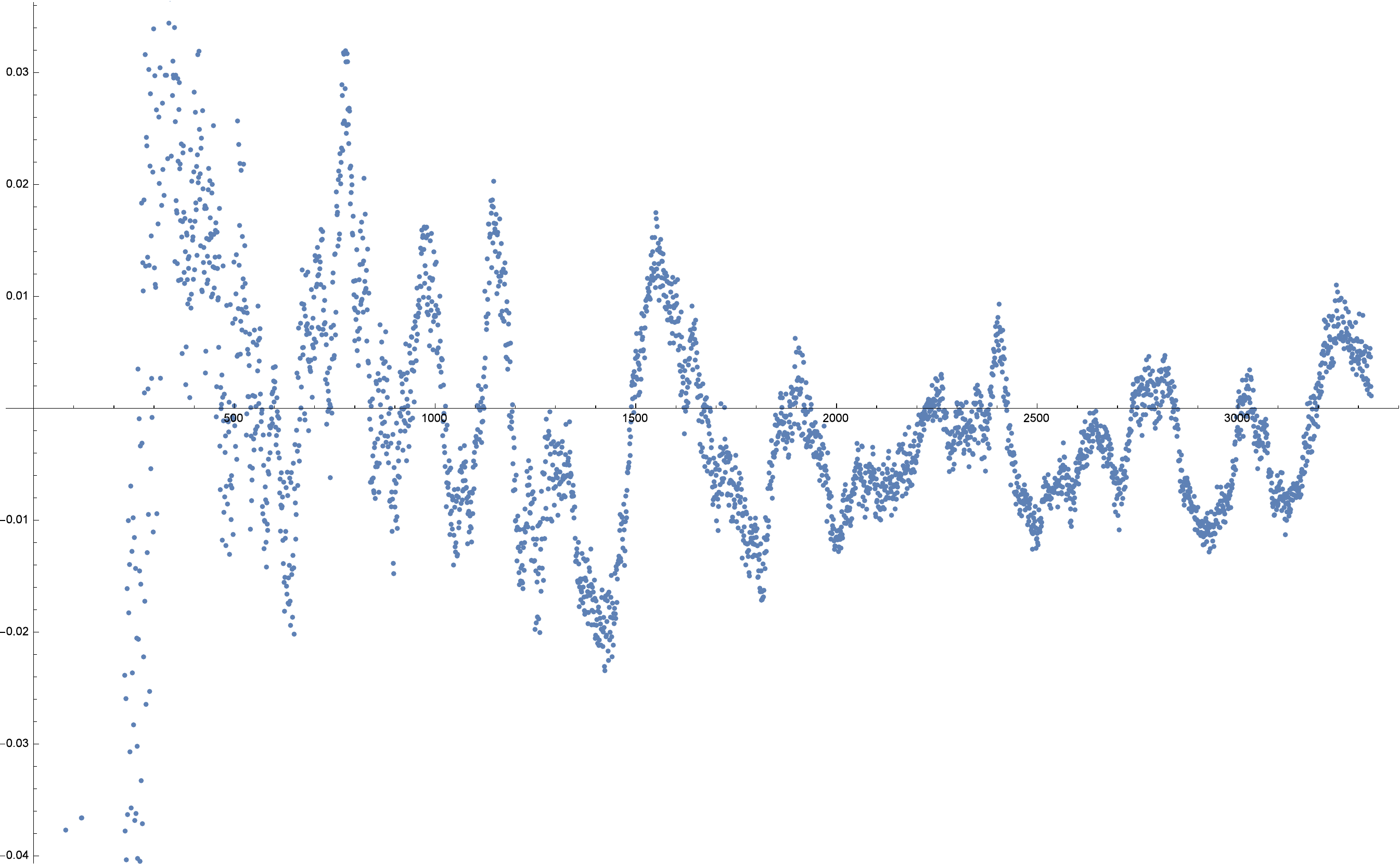}
        \includegraphics[width=\textwidth]{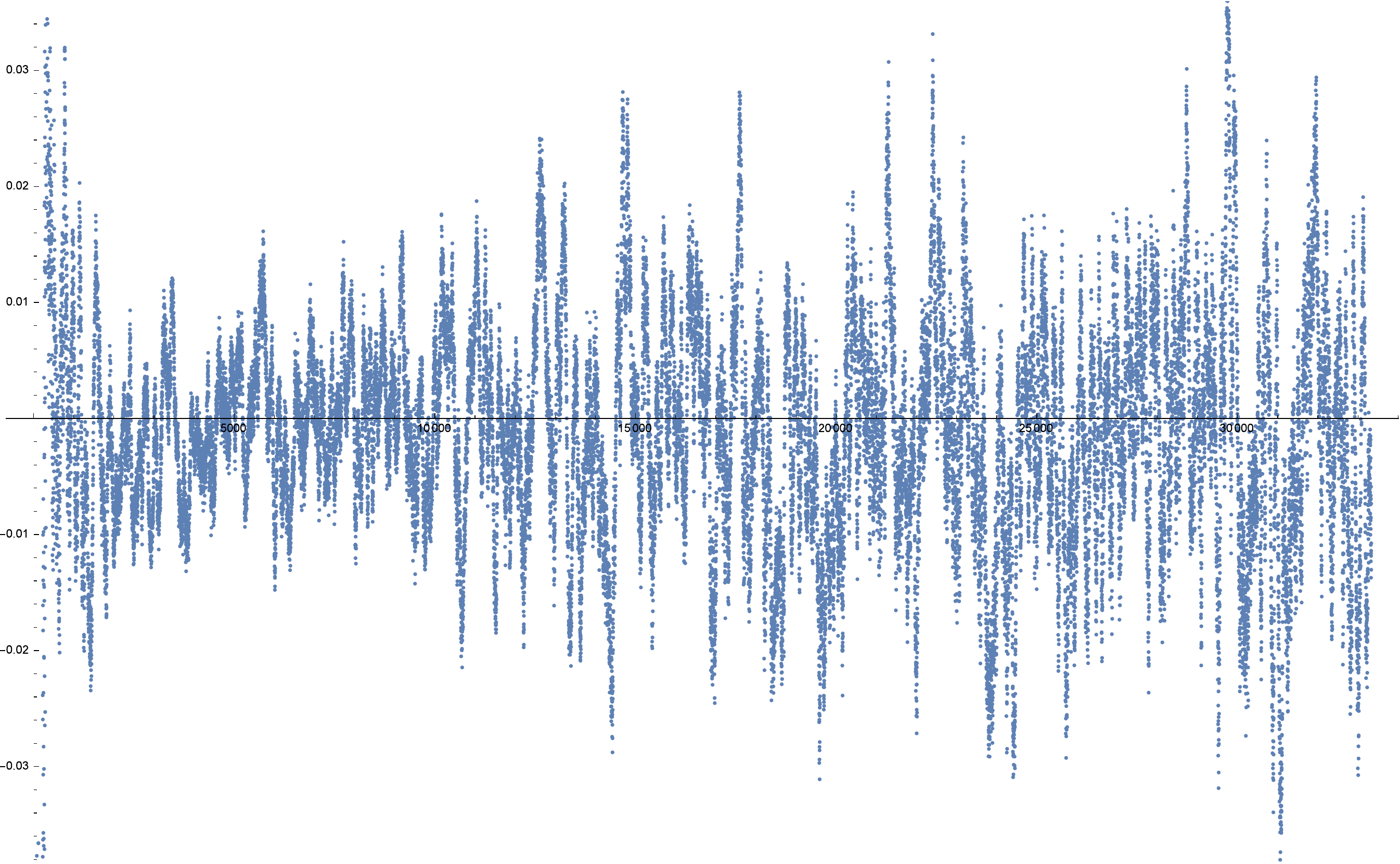}
    \caption{Top: A plot of estimated $(tf(1,2,d) - 1/2) \cdot d^{1/2}$ for $d \equiv 0 \pmod{3},d \leq 10000.$ Bottom: A plot of the same function for $1 \leq d \leq 100000.$ }
    \label{fig:tf12d-est1}
\end{figure}

\begin{figure}[htbp]
	\centering
    \includegraphics[width=0.9\textwidth]{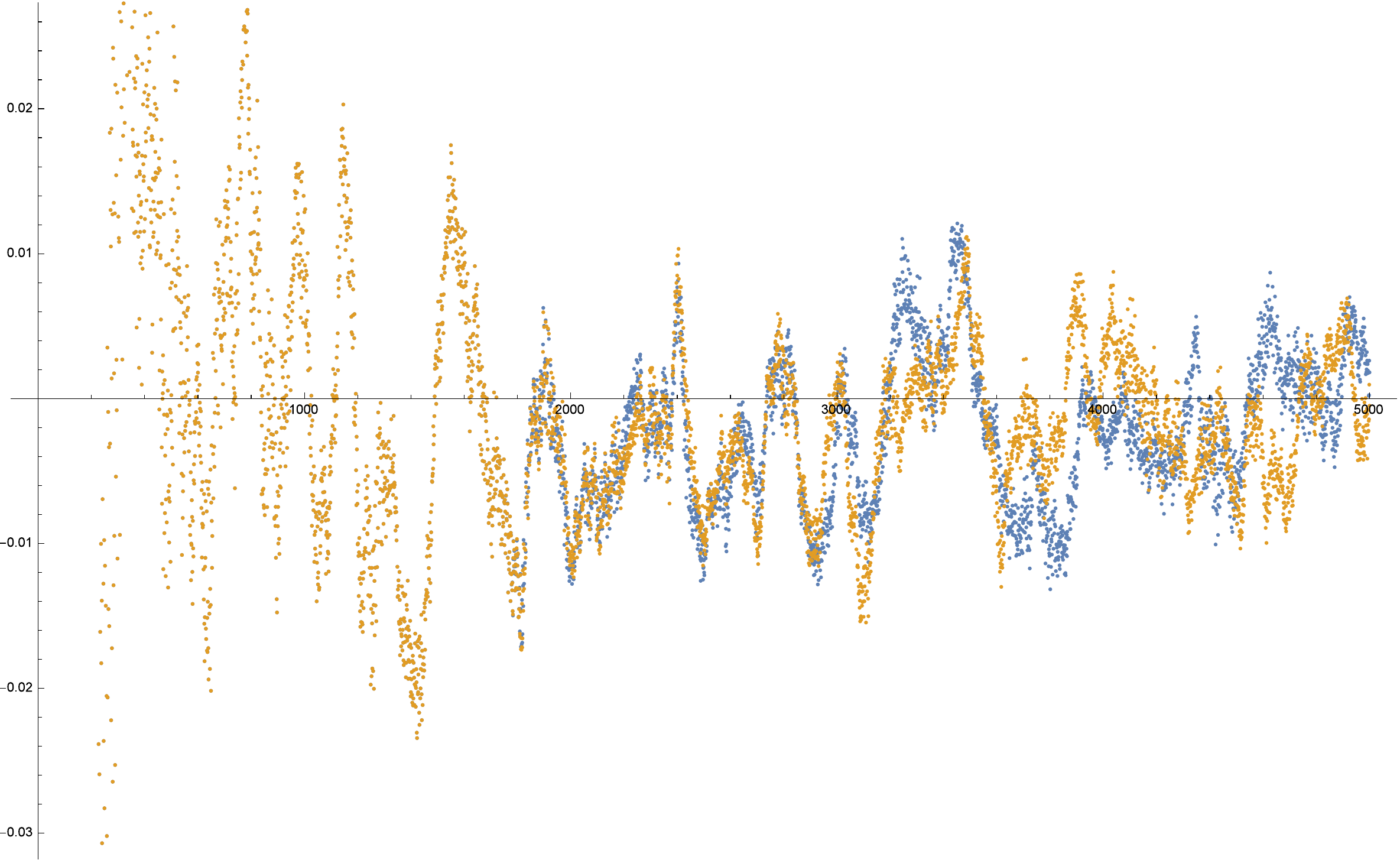}
    \caption{A plot of estimated $tf(1,2,d) \cdot d^{1/2}$ for $d \equiv 0 \pmod{3}, d \leq 15000$ for two separate trials.  Blue is one trial, and orange is another trial. The two series agree below $x = 1700, d = 5100$.}
    \label{fig:tf12d-est2}
\end{figure}

\begin{figure}[htbp]
	\centering
    \includegraphics[width=0.9\textwidth]{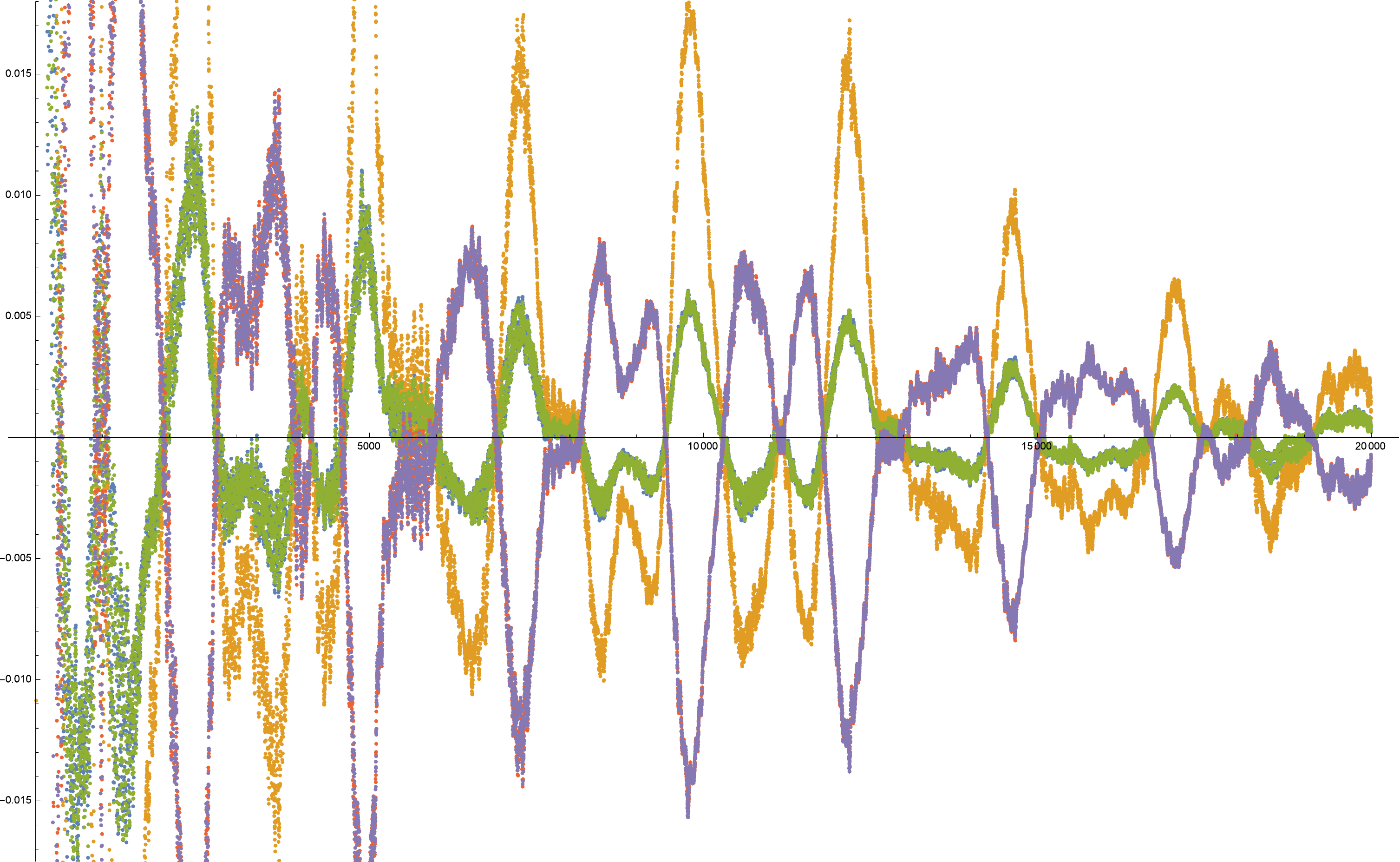}
    \caption{A plot of estimated $tf(2,3,d)$ for $d \leq 100000.$ Blue, orange, green, red, and purple differentiate points where $d$ is $0, 1, 2, 3, $ and $4 \pmod{5},$ respectively.  
    }
    \label{fig:tf23d-est}
\end{figure}

\begin{figure}[htbp]
\centering
    \includegraphics[width=\textwidth]{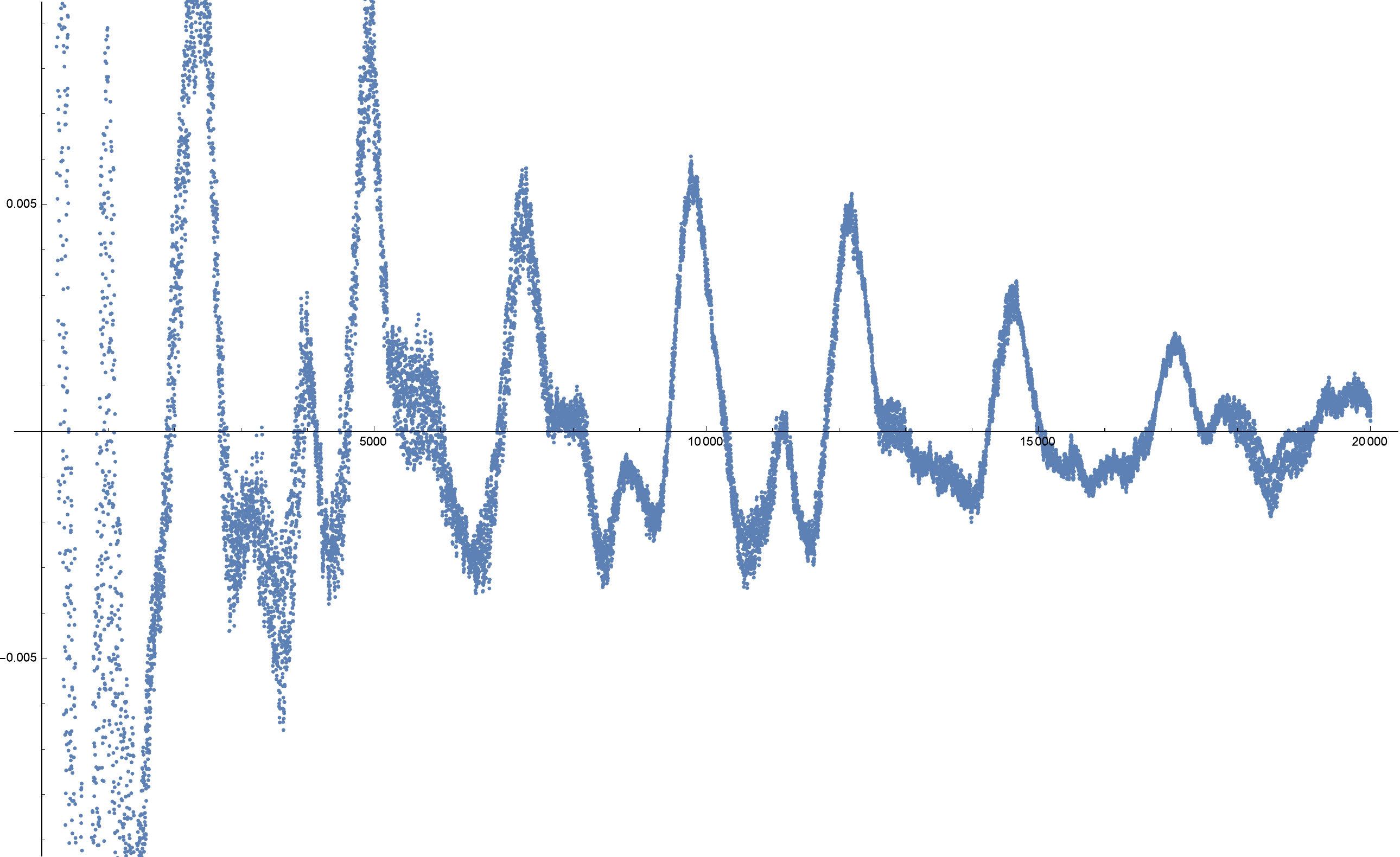}
  \includegraphics[width=\textwidth]{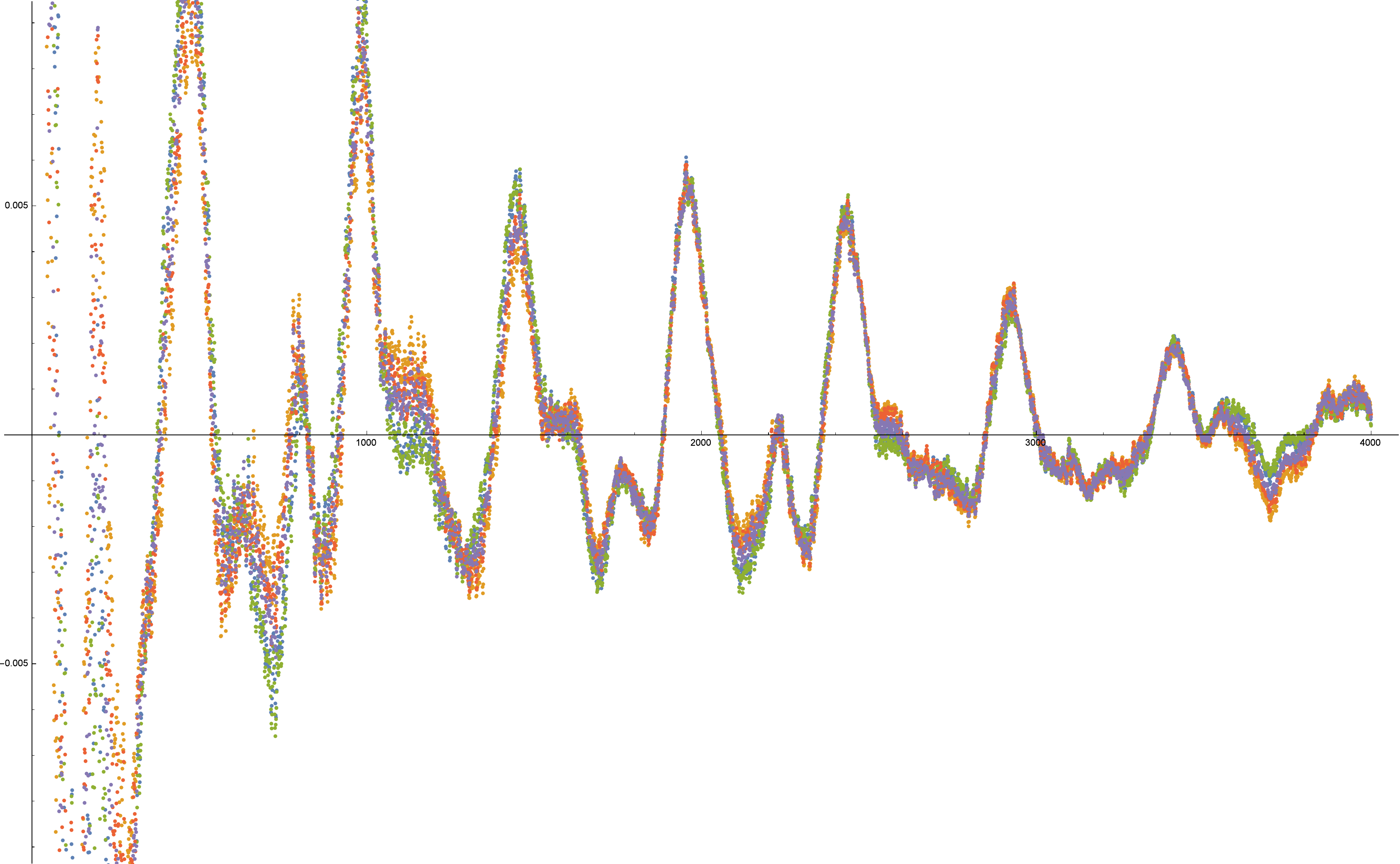}
    \caption{Top: A plot of estimated $tf(2,3,d)$ for $d \equiv 0 \pmod{5}, d \leq 100000$. Bottom: A plot of the same series, except the colors differentiate the $5$ sub-residue classes $\text{mod}$ 25.
    }
    \label{fig:tf23d-est2}
\end{figure}

\begin{figure}[htbp]
\centering
\includegraphics[width=\textwidth]{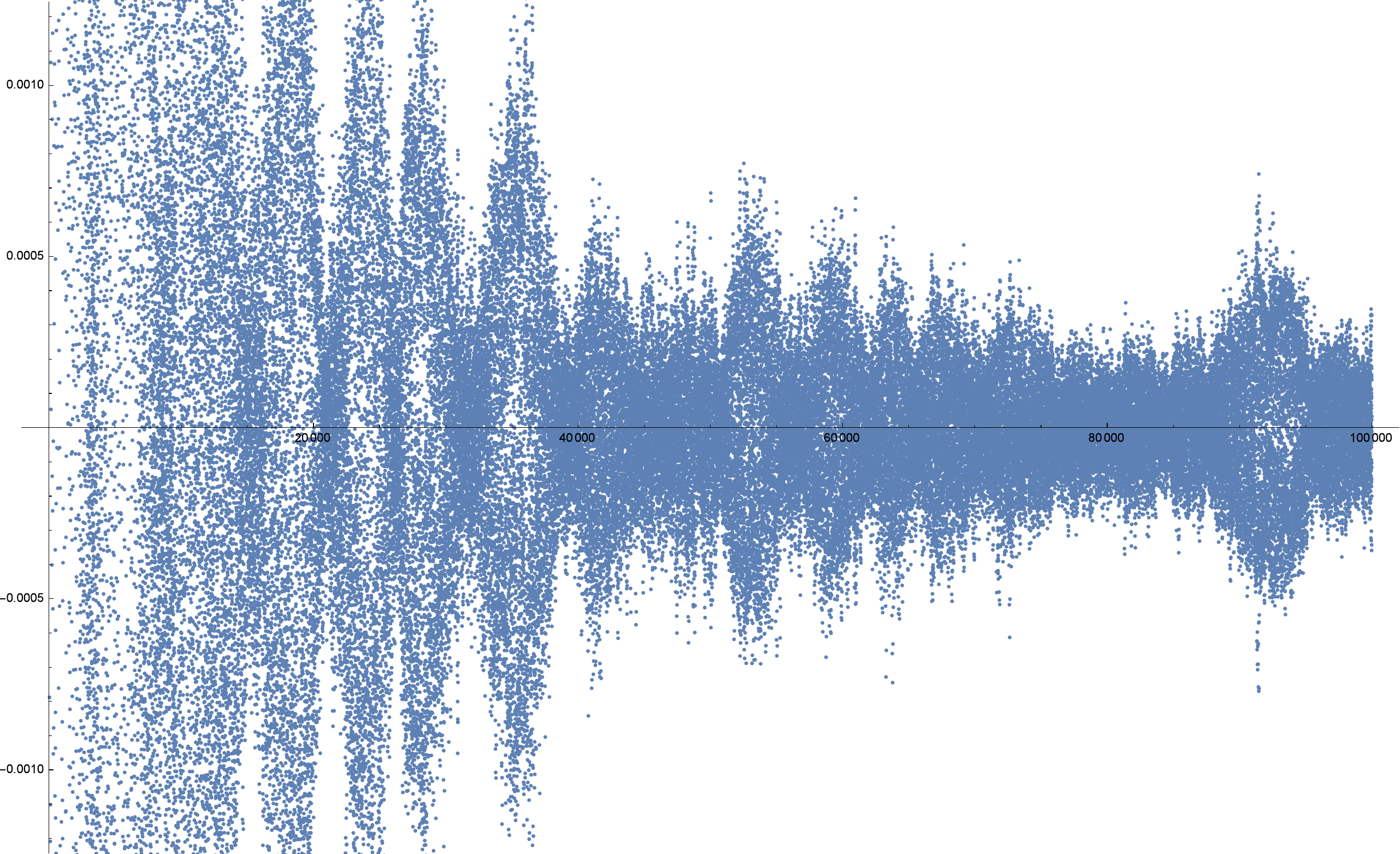}
\includegraphics[width=\textwidth]{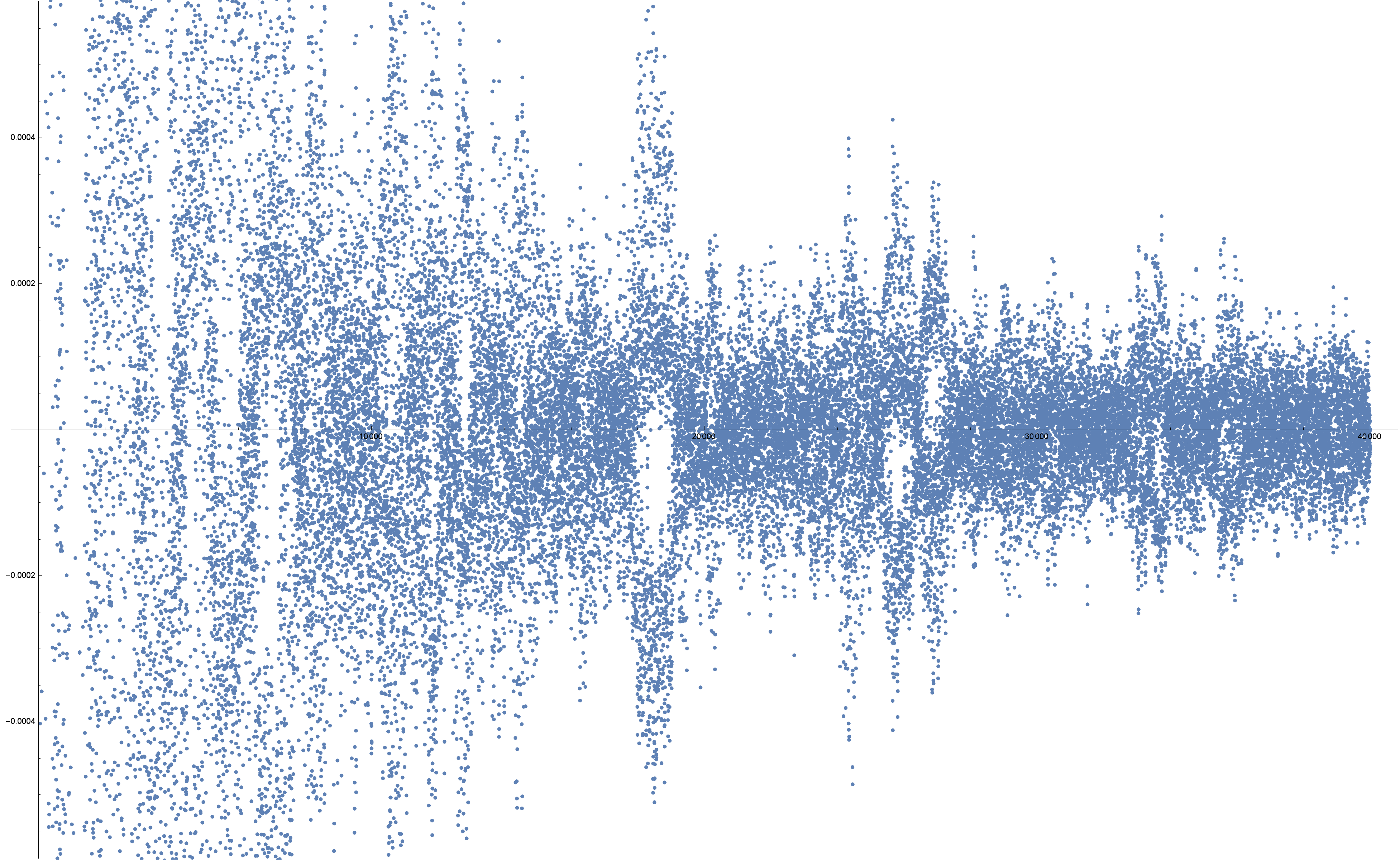}
\caption{Top: A plot of estimated $h(d) := tf(2,3,d) -  2 \cos(2 \pi / 5) tf(2,3,d+1) + tf(2, 3, d+2)$ for $d \leq 100000.$ Bottom: The same series, except restricted to $d \equiv 0 \pmod{5}$ and extended to $d \leq 200000.$
}
\label{fig:tf23d-res1}
\end{figure}

\begin{figure}[htbp]
\centering
\includegraphics[width=\textwidth]{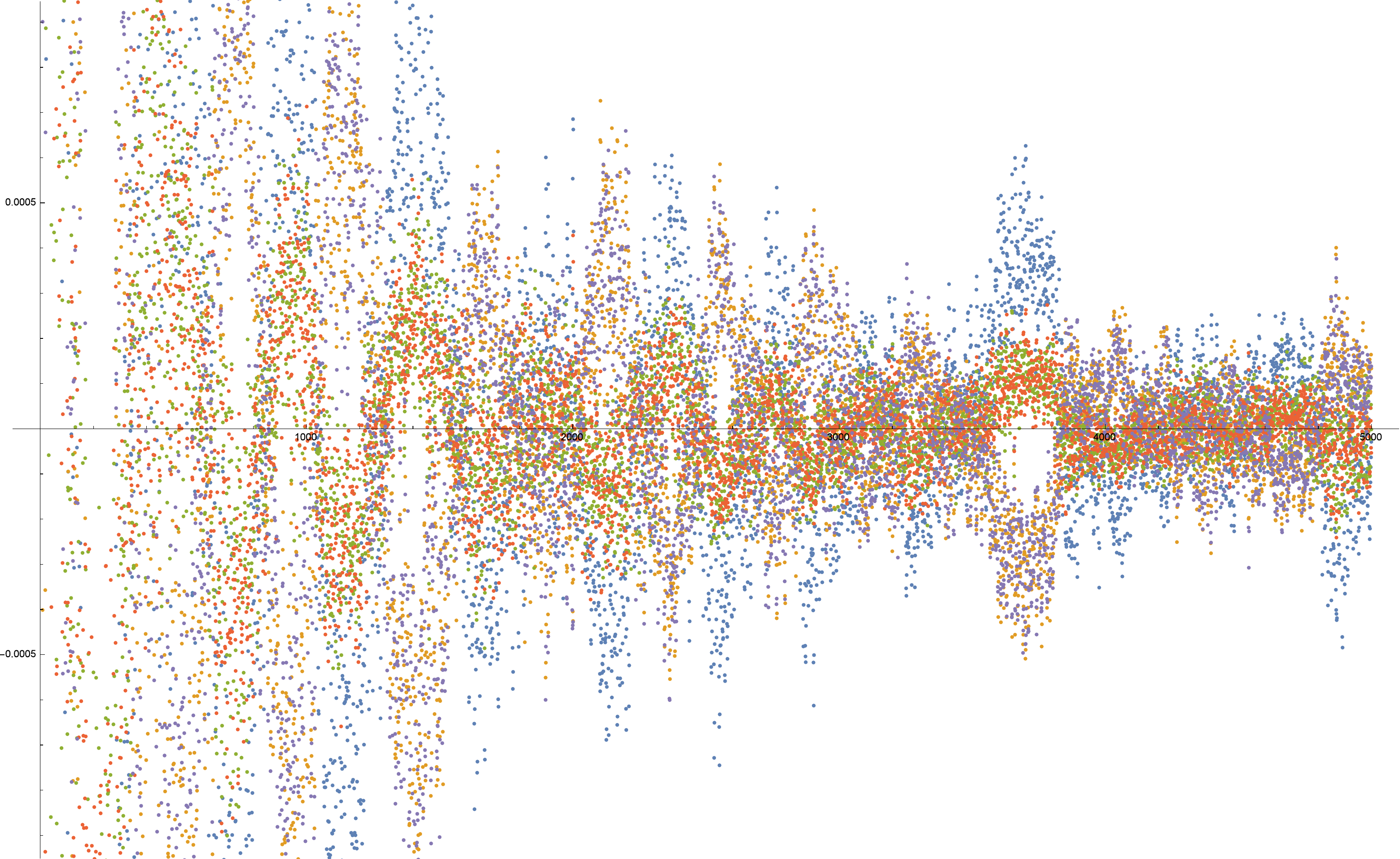}
\includegraphics[width=\textwidth]{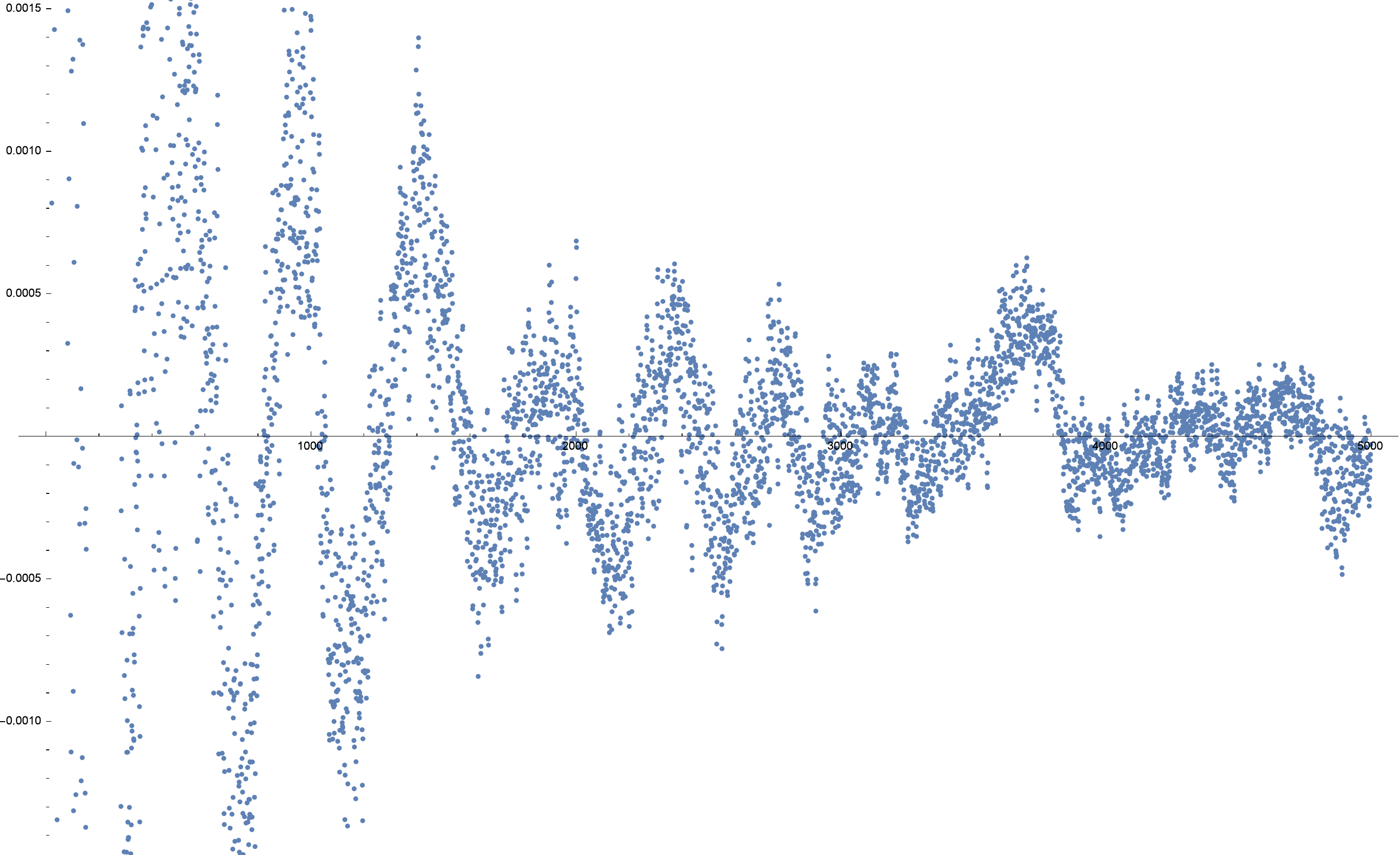}
\caption{Top: A plot of $g(c) := tf(2, 3, 5c) - 2 \cos(2 \pi / 5) tf(2,3,5c+1) + tf(2, 3, 5c+2)$ for $c \leq 25000$, and the five colors correspond to residue classes $\text{mod}$ 5. Bottom: we show just one of these classes.
}
\label{fig:tf23d-res2}
\end{figure}

\begin{figure}[htbp]
\centering
\includegraphics[width=\textwidth]{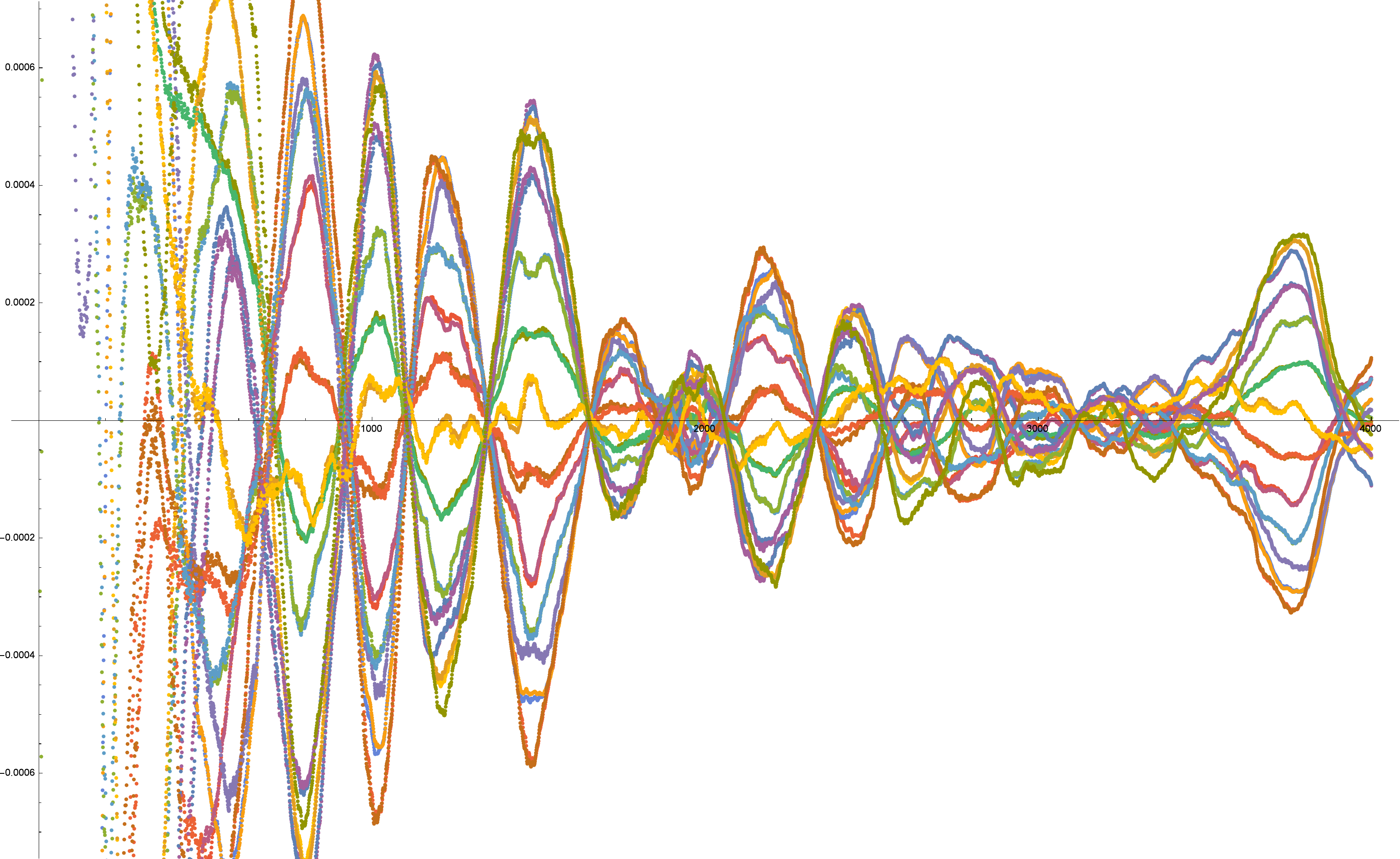}
\includegraphics[width=\textwidth]{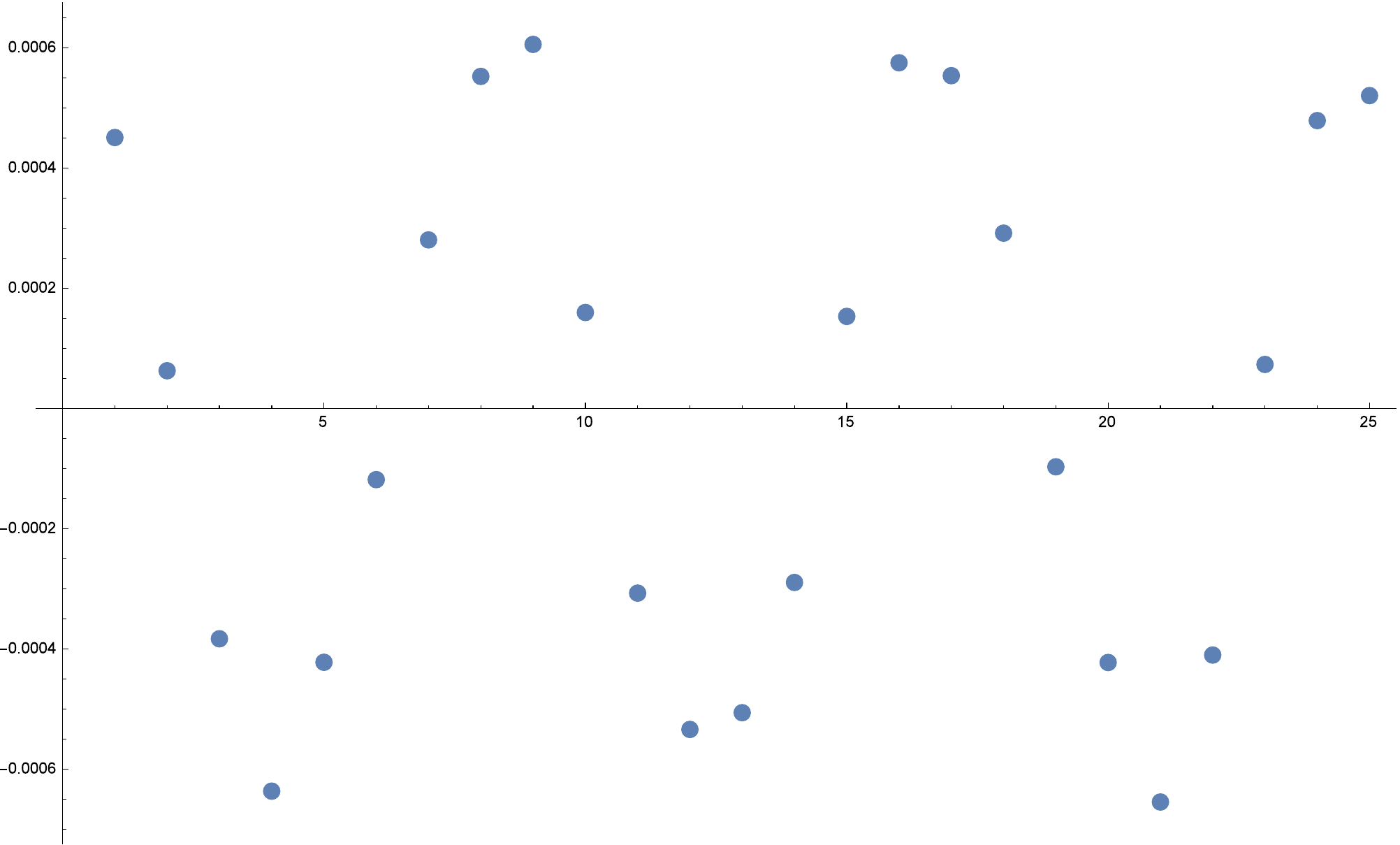}
\caption{Top: A plot of $h(d) = tf(2,3,d) -  2 \cos(2 \pi / 5) tf(2,3,d+1) + tf(2, 3, d+2)$ split into residues class $\text{mod}$ 25, and in each class, we take an exponential moving average with time constant $100.$ Bottom: A plot of $h(25000+d)$ vs $d$ for $1 \le d \le 25.$
}
\label{fig:tf23d-res3}
\end{figure}

\begin{figure}[htbp]
	\centering
    \includegraphics[width=\textwidth]{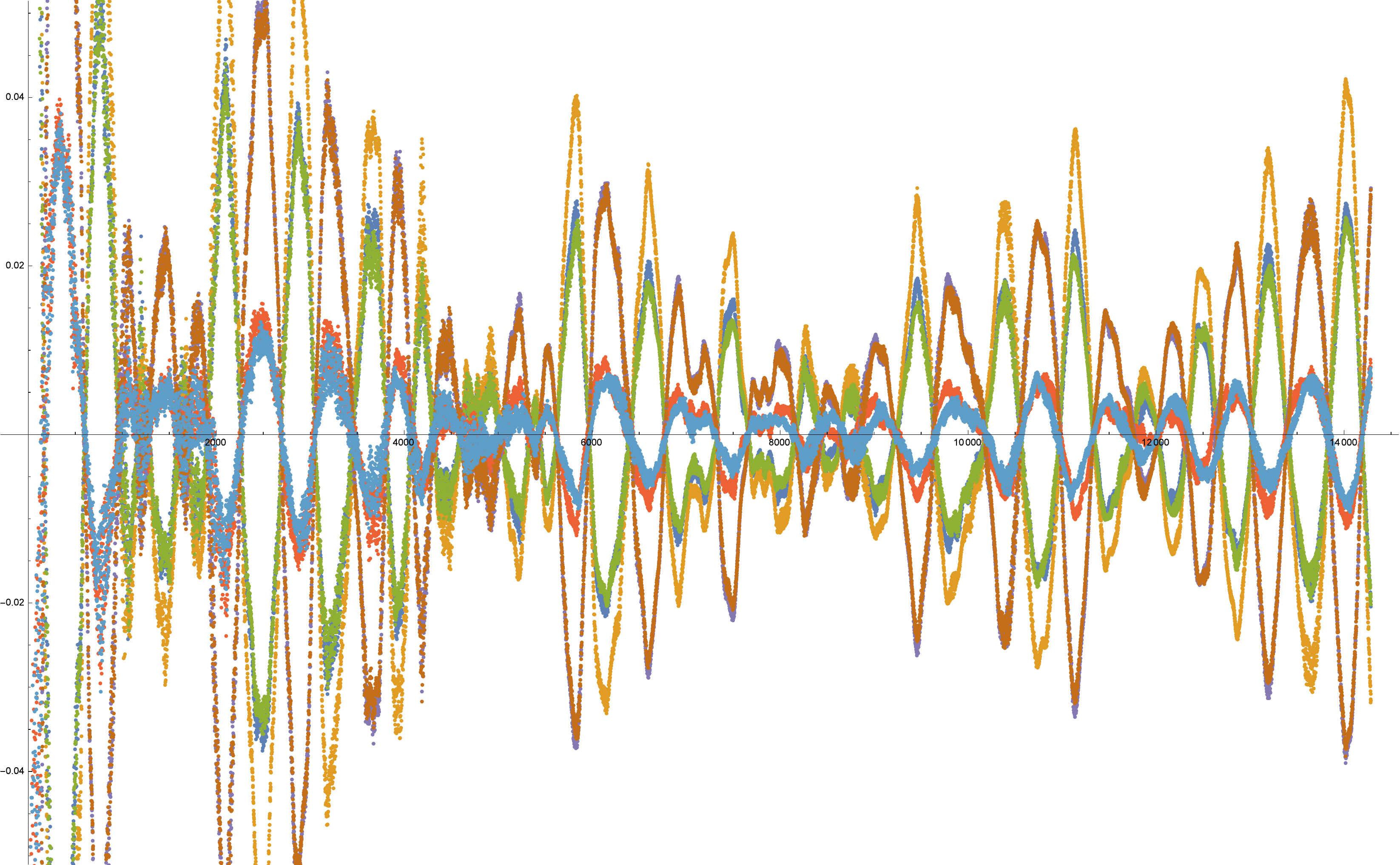}
    \caption{A plot of estimated $tf(3,4,d)$ for $d \leq 100000.$ Blue, orange, green, red, and purple, orange-red, and baby blue differentiate points where $d$ is $0, 1, 2, 3, 4, 5$ and $6 \pmod{7},$ respectively. 
    }
    \label{fig:tf34d-est}
\end{figure}

\begin{figure}[htbp]
	\centering
    \includegraphics[width=\textwidth]{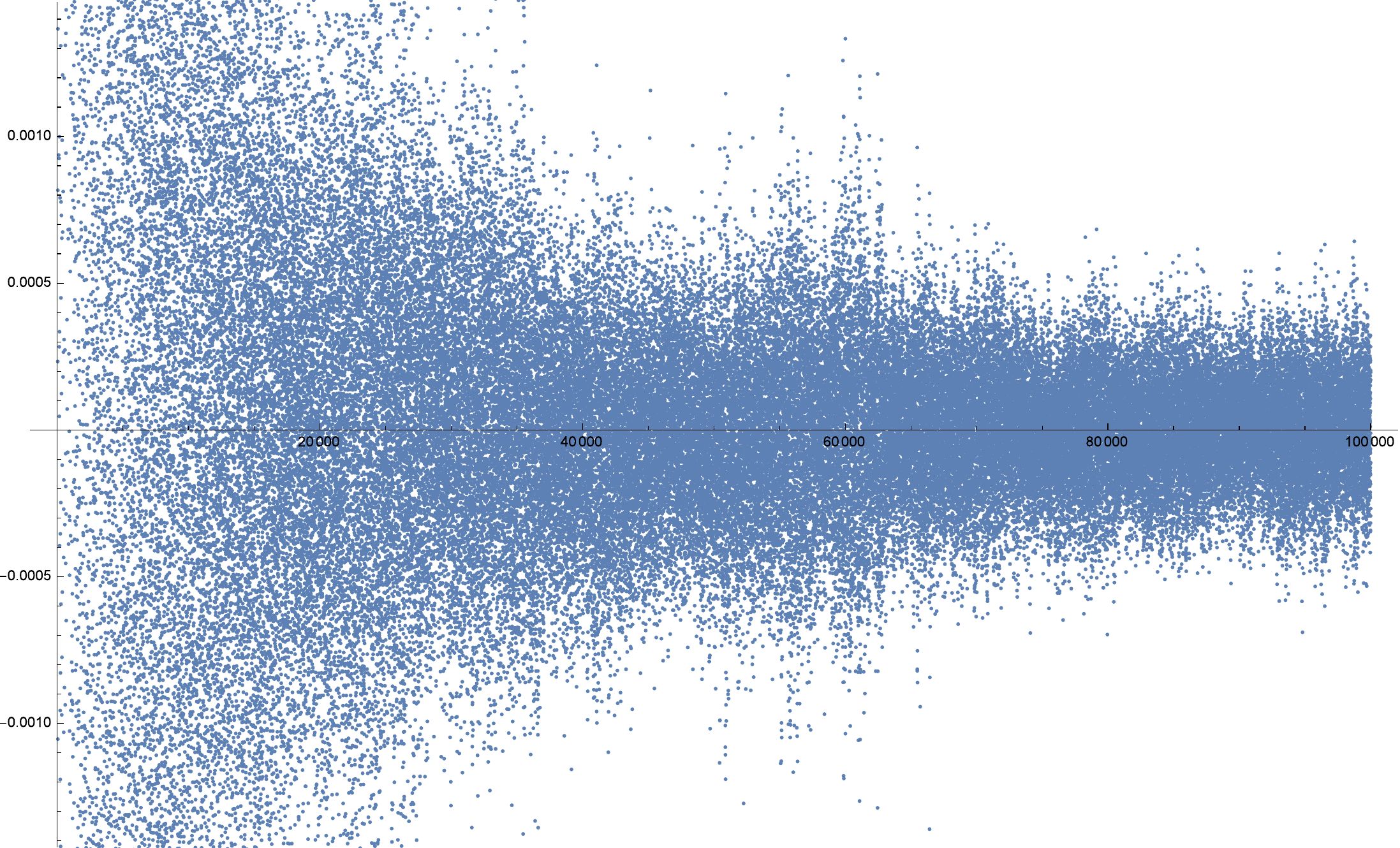}
    \caption{A plot of  $tf(1,4,d)$ for $d \leq 100000.$
    }
    \label{fig:tf14d-est}
\end{figure}

\begin{figure}[htbp]
	\centering
    \includegraphics[width=\textwidth]{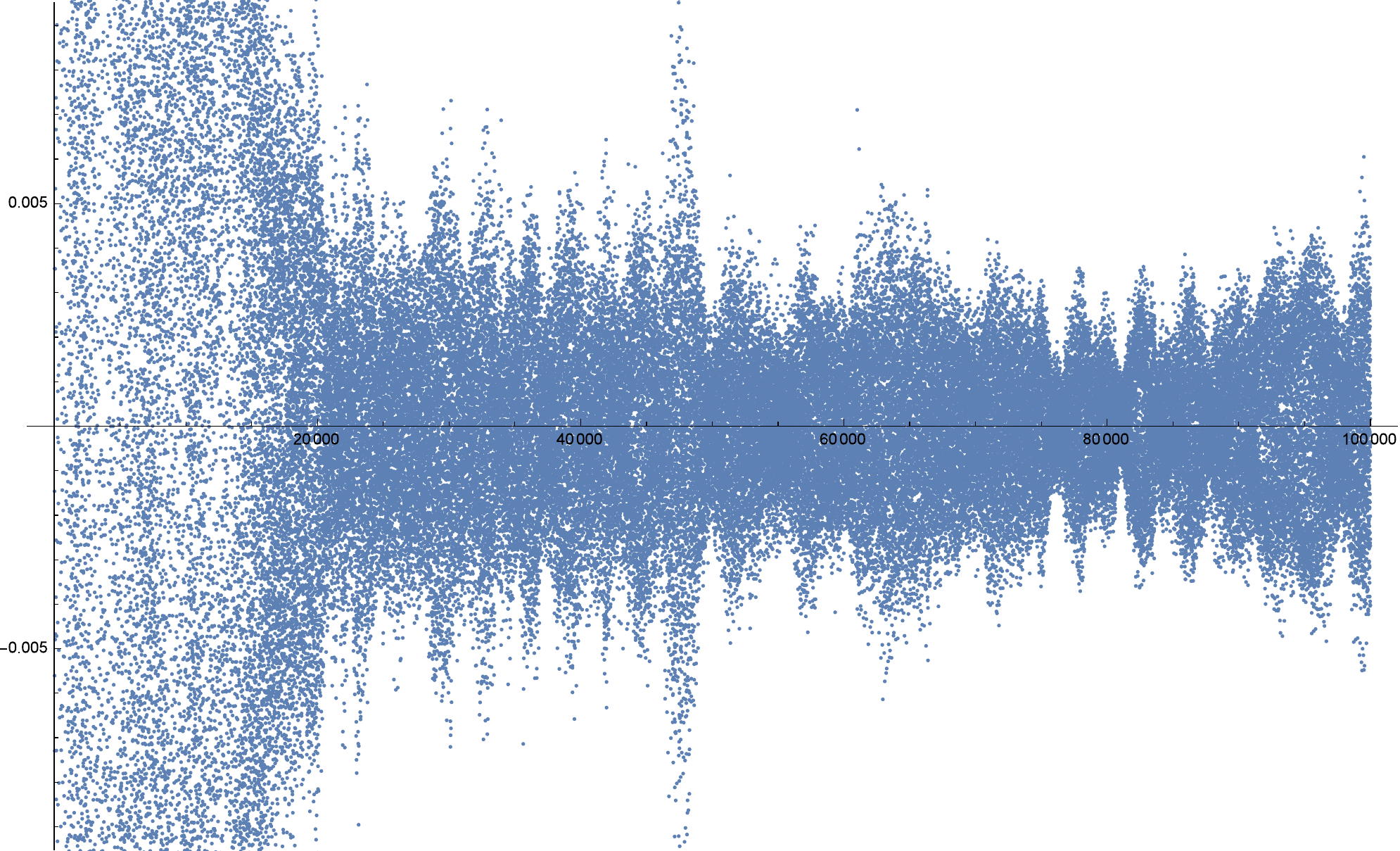}
    \caption{A plot of  $tf(1,6,d)$ for $d \leq 100000.$
    }
    \label{fig:tf16d-est}
\end{figure}

\begin{figure}[htbp]
	\centering
    \includegraphics[width=\textwidth]{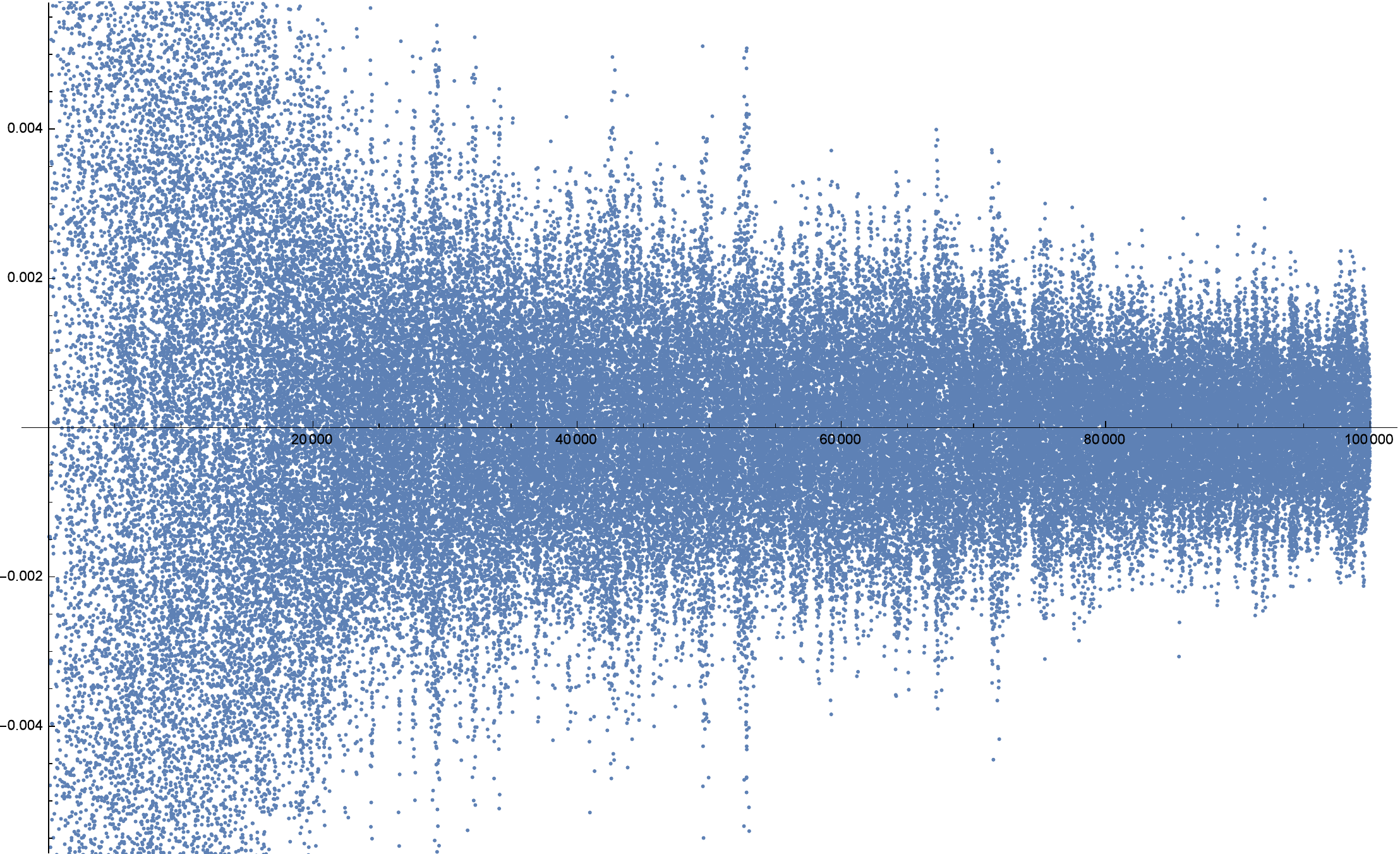}
    \caption{A plot of $tf(2,5,d)$ for $d \leq 100000.$
    }
    \label{fig:tf25d-est}
\end{figure}

\begin{figure}[htbp]
\centering
\includegraphics[width = 0.48 \textwidth]{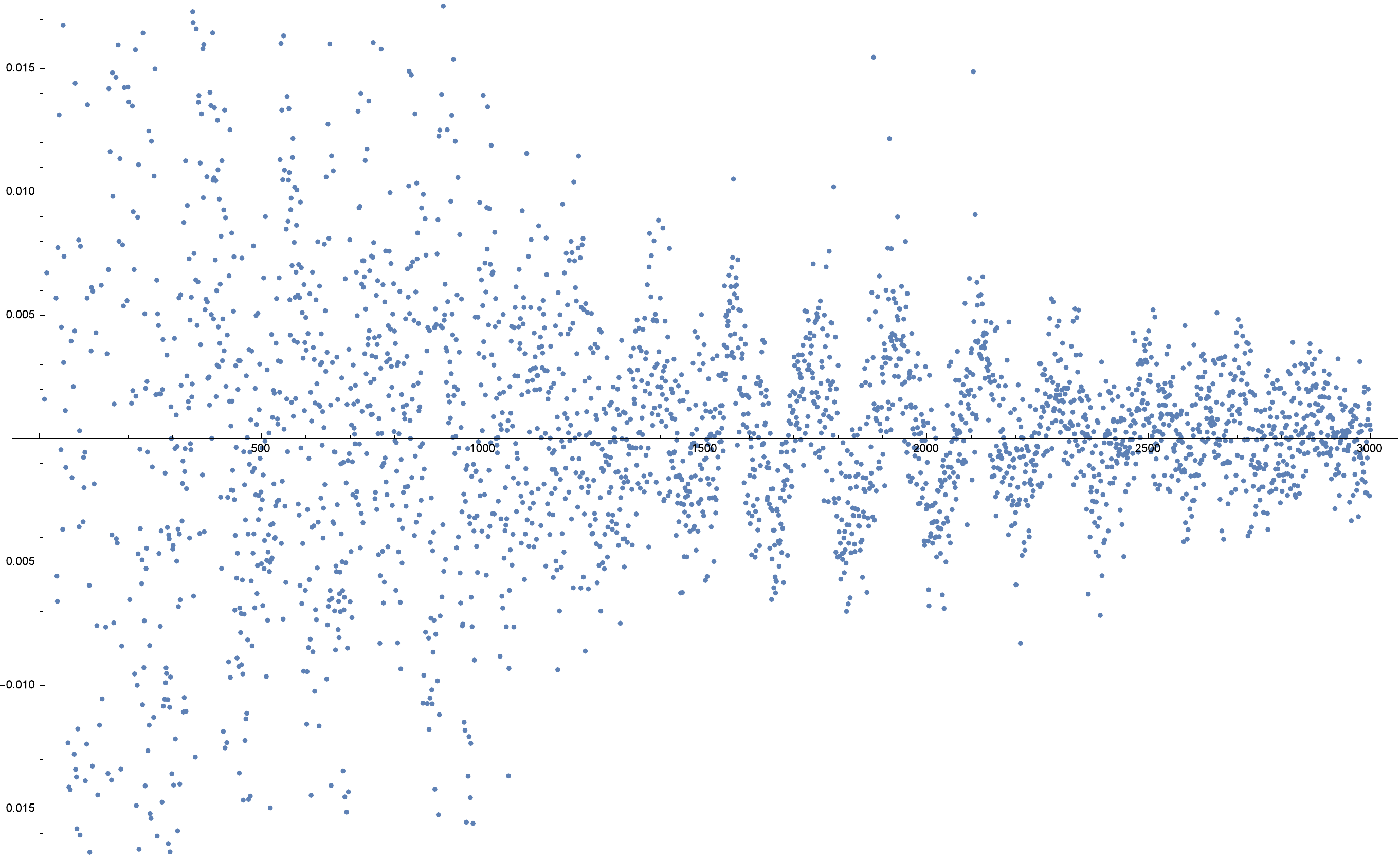}
\includegraphics[width = 0.48 \textwidth]{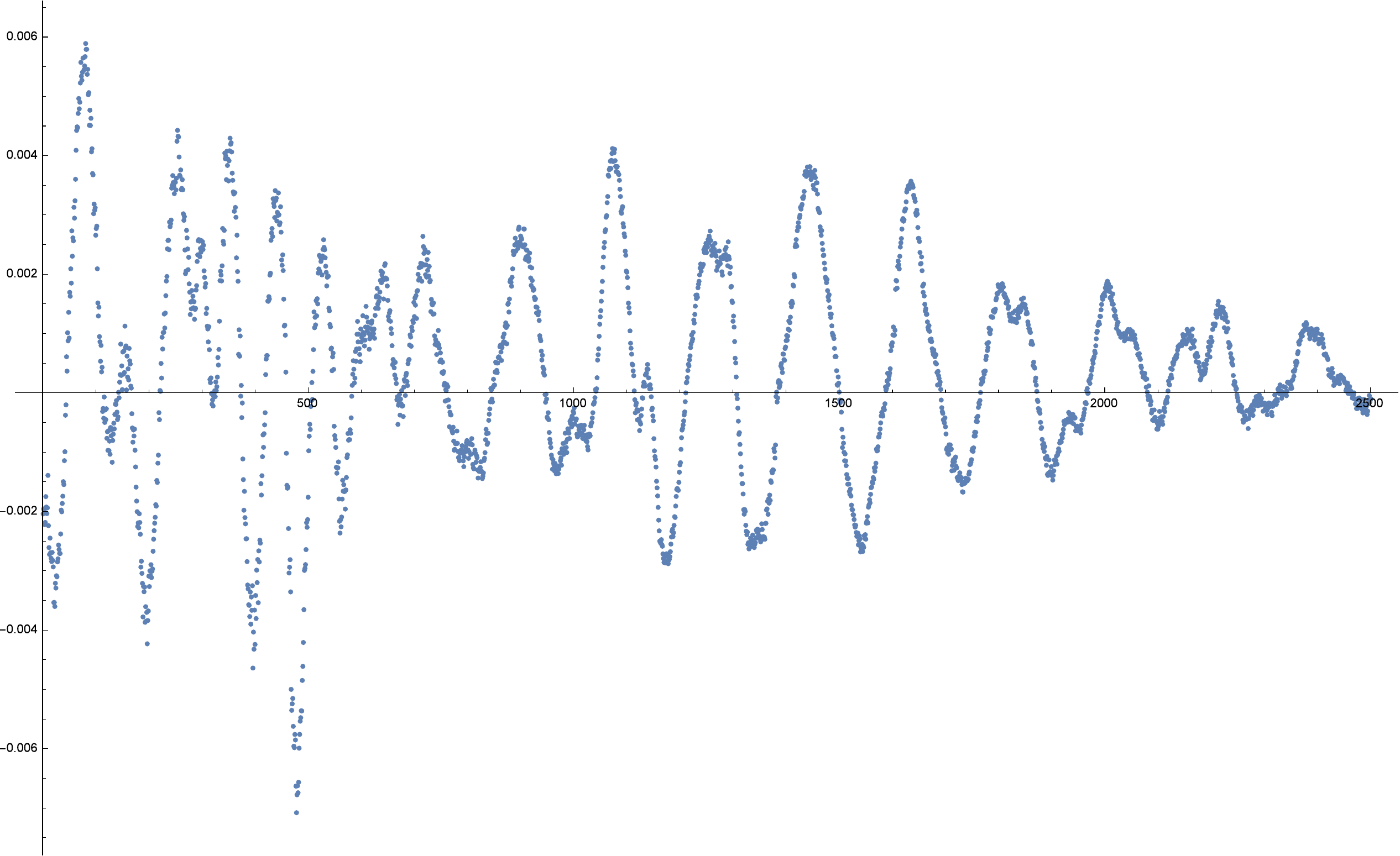}
\includegraphics[width = \textwidth]{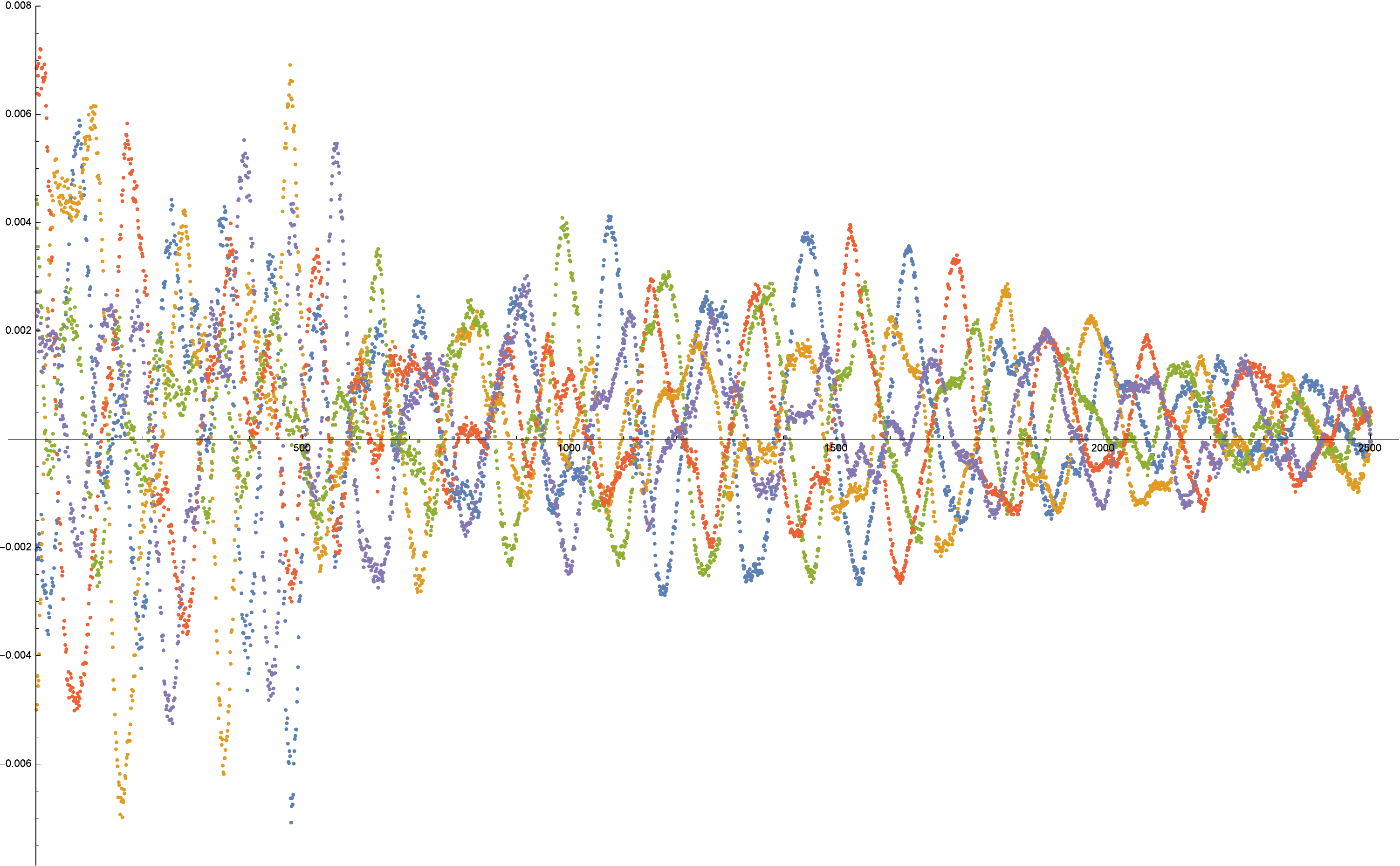}
\includegraphics[width = 0.48 \textwidth]{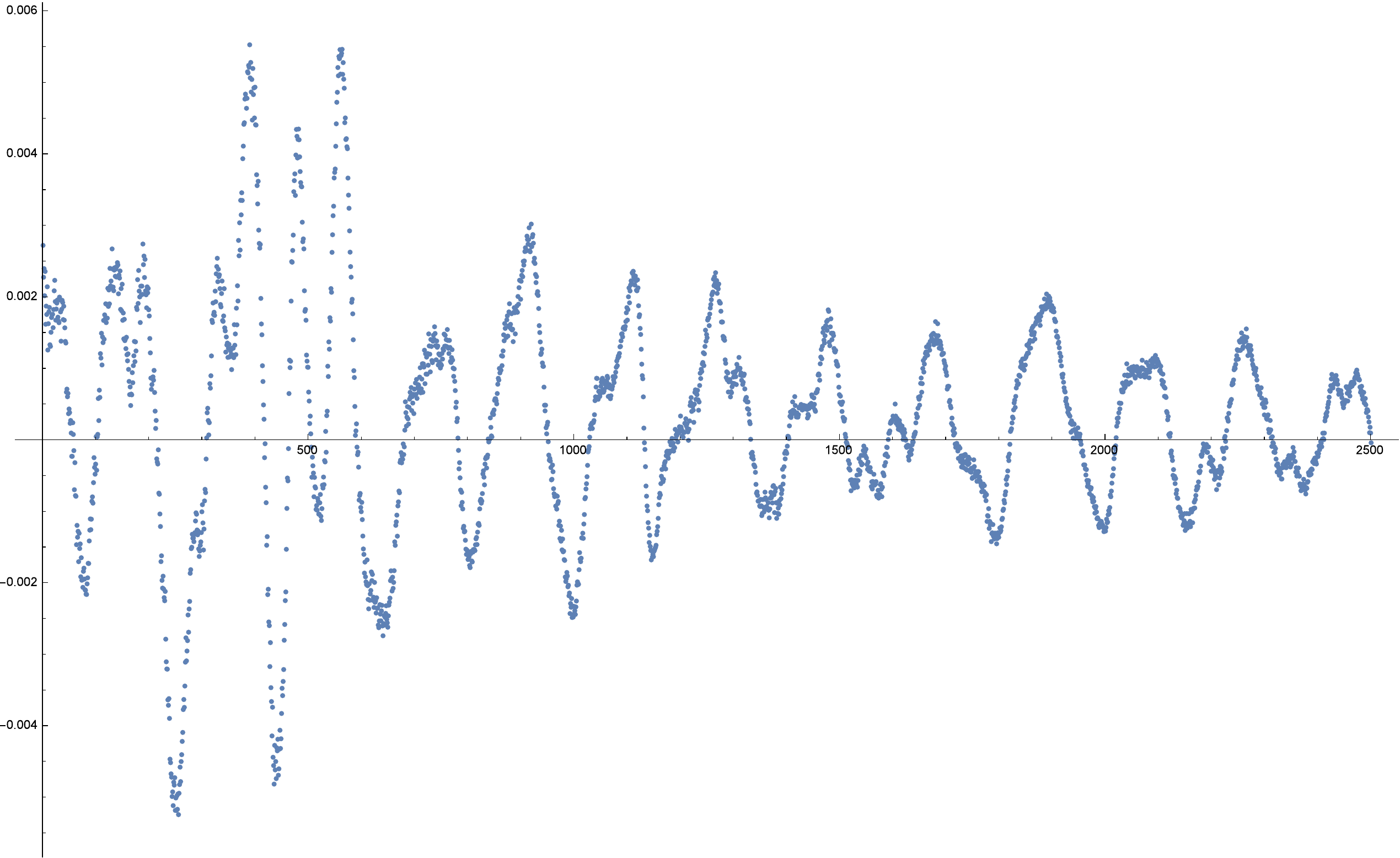}
\caption{First plot: A plot of $tf(1,4,d)$ restricted to $d \equiv 0 \pmod{5}, 0 \le d \le 15000.$ Second: The exponential mean of the first plot with time constant 20 and $2500 \le d \le 15000.$ Third: Like the second plot, but with all five classes $\text{mod}$ $5$. Fourth: Like the second plot, but $d \equiv 4 \pmod{5}.$
}
\label{fig:tf14d-est2}
\end{figure}

\begin{figure}[htbp]
\centering
\includegraphics[width = \textwidth]{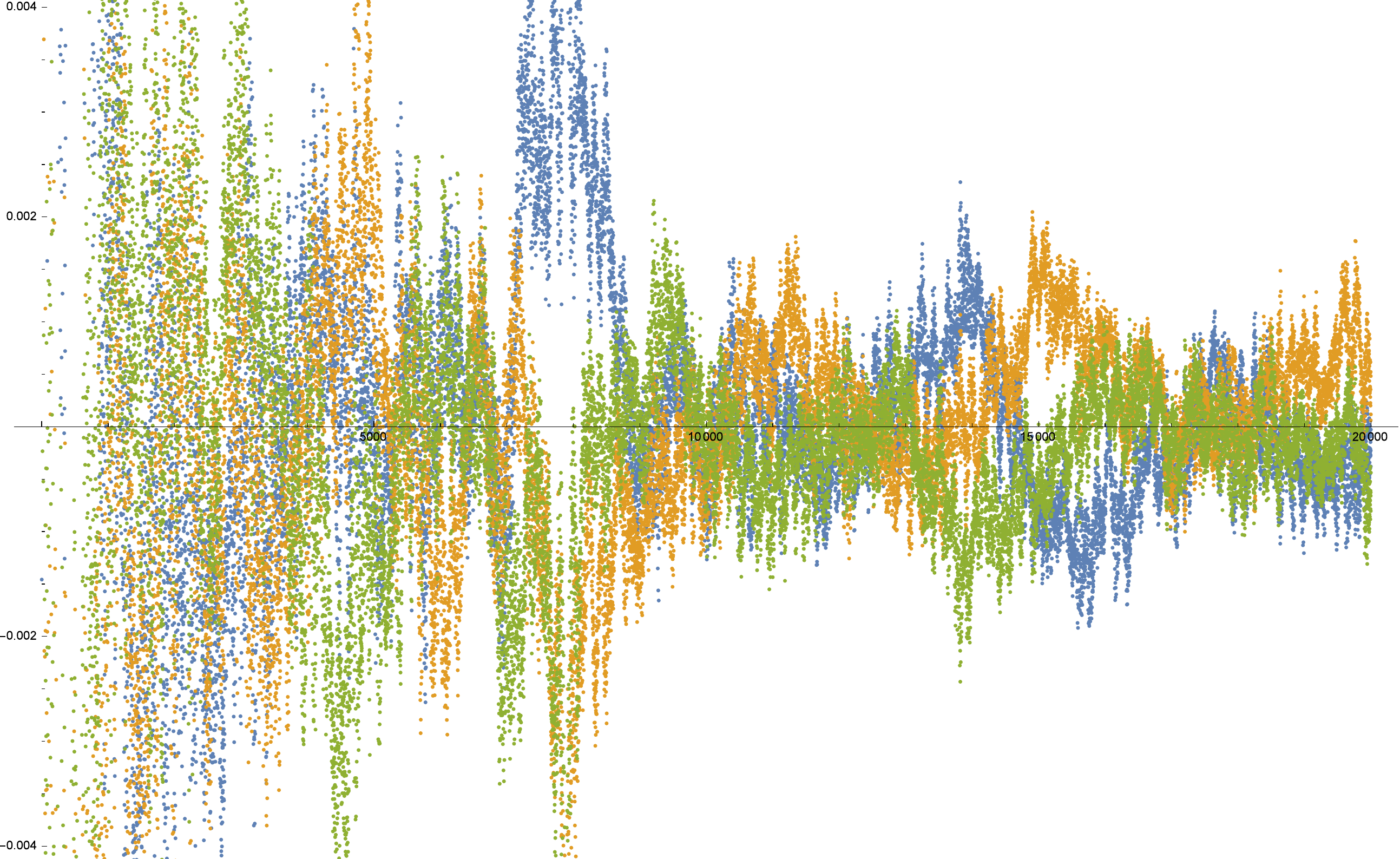}
\includegraphics[width = \textwidth]{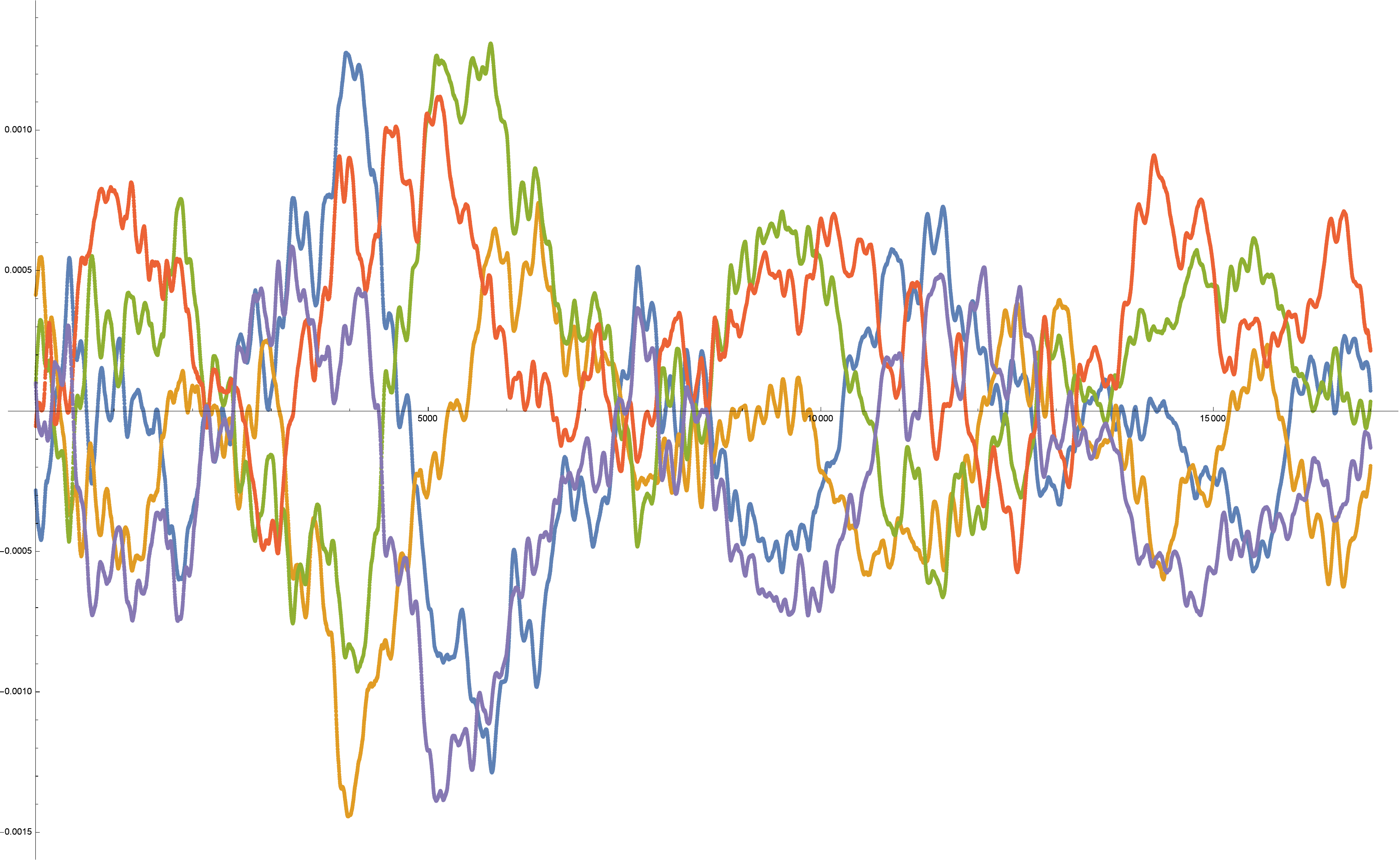}
\caption{Top: A plot of $tf(1,6,d)$ restricted to $d \equiv 0 \pmod{7}, 0 \le d \le 210000.$ The three colors distinguish the three subclasses $\text{mod}$ $21.$ Bottom: A plot of $tf(1,6,d)$ restricted to $210000 \le d \le 567000$ and $Mod(d,21) \in \{0, 4, 6, 8, 12\},$ after an exponential mean with time constant $100$ is applied to each class separately.
}
\label{fig:tf16d-est2}
\end{figure}

\begin{figure}[htbp]
\centering
\includegraphics[width = 0.67\textwidth]{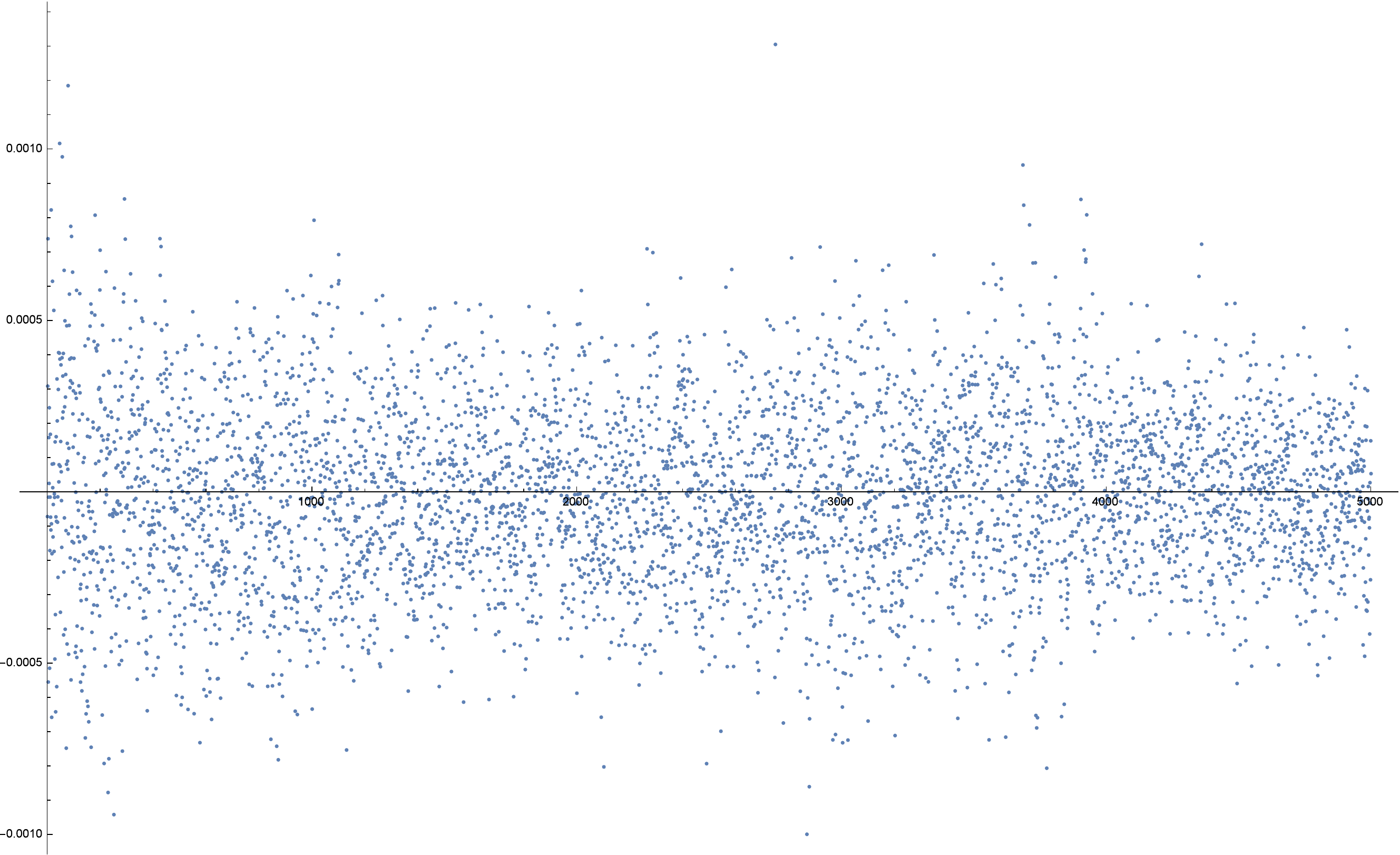}
\includegraphics[width = 0.67\textwidth]{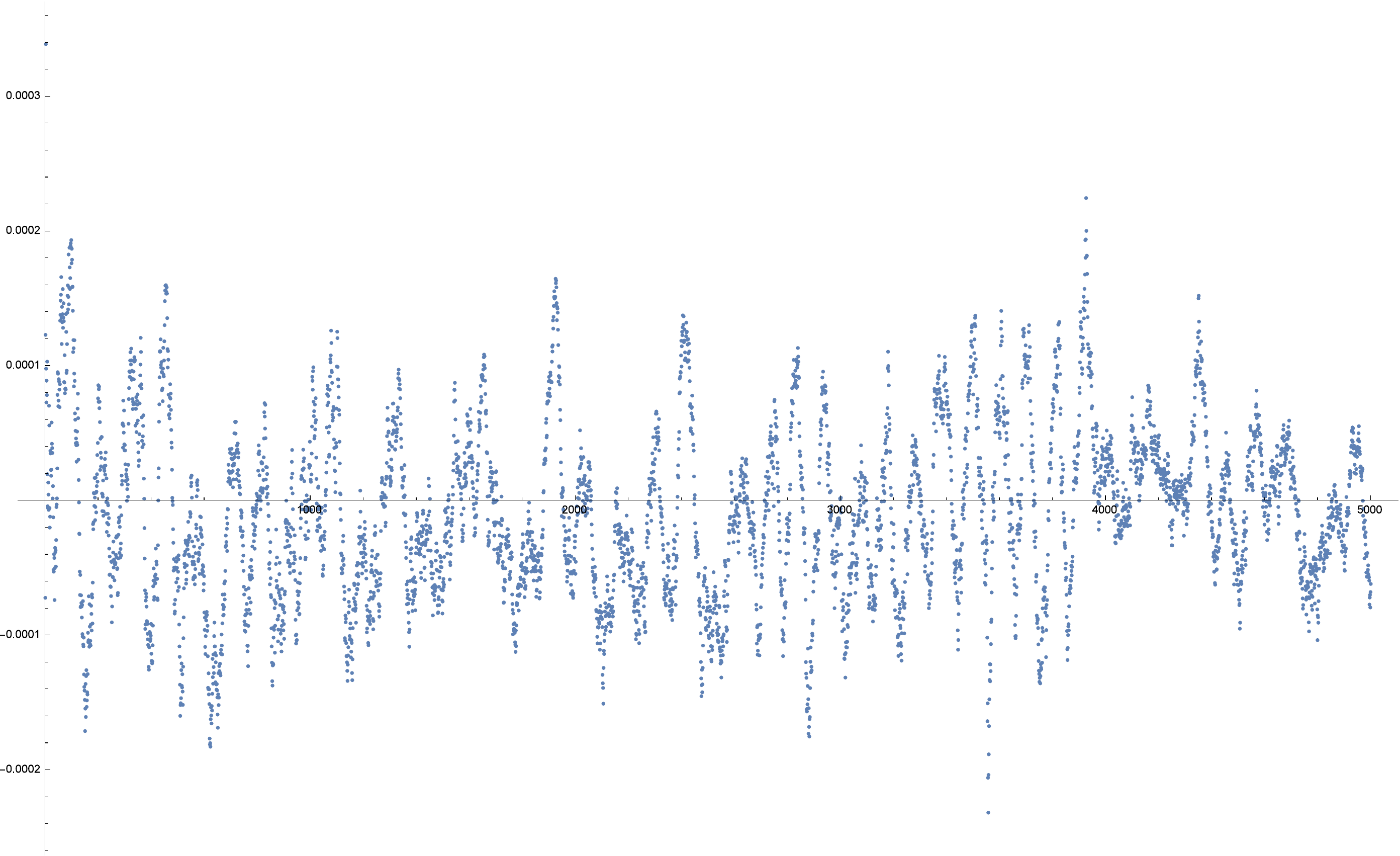}
\includegraphics[width = 0.67\textwidth]{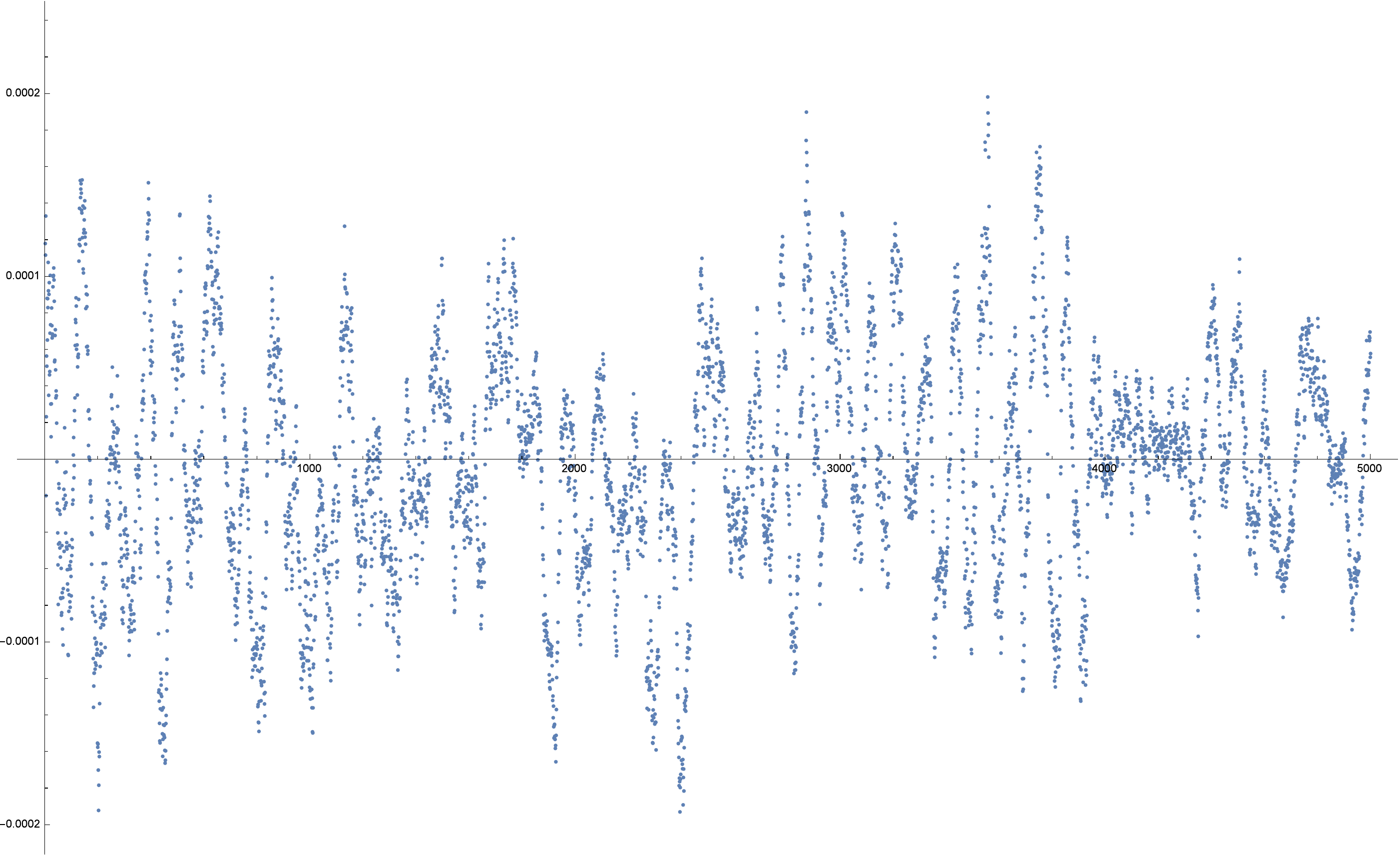}
\caption{First plot: $tf(2,5,d)$ restricted to $d \equiv 0 \pmod{7}, 0 \le d \le 100000.$ Second plot: A moving exponential mean of this series with time constant $20.$ Third plot: Like the second plot, but $d \equiv 3 \pmod{7}.$
}
\label{fig:tf25d-est2}
\end{figure}

\section{Acknowledgements}
The author originally studied the Kolakoski sequence as a part of the ``Math Project Lab'' course at MIT with Yongyi Chen and Michael Yan in Spring 2016. 
The notions $E_{m, n},$ $C_{m, n},$ and several smaller elements of this paper are based on work from this project. The unusual patterns in the correlation function for $d \le 500$ were also observed during thie time.

\end{document}